\documentclass[11pt]{article}

\usepackage{amsmath, amssymb, amsthm, verbatim,enumerate,bbm,color,mathtools} 
\usepackage{subfig}
\usepackage{mathrsfs}
\usepackage{enumitem}
\usepackage{graphicx}
\usepackage{epstopdf}
\usepackage{svg}
\usepackage{caption}
\usepackage{subcaption}
\usepackage{thmtools}
\usepackage{thm-restate}
\usepackage{float}

\usepackage[colorlinks,citecolor=blue,bookmarks=true]{hyperref}

\title{When does a tree activate the random graph?}

\author{Asaf Cohen Antonir\thanks{School of Mathematical Sciences, Tel Aviv University, Tel Aviv 69978, Israel.} \and Yuval Peled\thanks{Einstein Institute of Mathematics, Hebrew University, Jerusalem 91904, Israel.} \and Asaf Shapira\footnotemark[1] \and Mykhaylo Tyomkyn\thanks{Department of Applied Mathematics, Faculty of Mathematics and Physics,
Charles University, Prague 11800, Czech Republic} \and Maksim Zhukovskii\thanks{School of Computer Science, The University of Sheffield, Sheffield S1 4DP, UK}}

\date{\today}
\allowdisplaybreaks

\theoremstyle{plain}
\newtheorem{theorem}{Theorem}[section]
\newtheorem{lemma}[theorem]{Lemma}
\newtheorem{claim}[theorem]{Claim}
\newtheorem{proposition}[theorem]{Proposition}

\newtheorem{observation}[theorem]{Observation}
\newtheorem{corollary}[theorem]{Corollary}

\newtheorem{problem}[theorem]{Problem}

\theoremstyle{remark}
\newtheorem{remark}[theorem]{Remark}

\theoremstyle{definition}
\newtheorem{definition}[theorem]{Definition}

\makeatletter
\def\moverlay{\mathpalette\mov@rlay}
\def\mov@rlay#1#2{\leavevmode\vtop{%
   \baselineskip\z@skip \lineskiplimit-\maxdimen
   \ialign{\hfil$\m@th#1##$\hfil\cr#2\crcr}}}
\newcommand{\charfusion}[3][\mathord]{
    #1{\ifx#1\mathop\vphantom{#2}\fi
        \mathpalette\mov@rlay{#2\cr#3}
      }
    \ifx#1\mathop\expandafter\displaylimits\fi}
\makeatother

\makeatletter
\renewenvironment{proof}[1][\proofname]
{\par\pushQED{\qed}
	\normalfont\topsep6\p@\@plus6\p@\relax\trivlist
	\item[\hskip\labelsep\bfseries#1\@addpunct{.}]
	\ignorespaces}
{\popQED\endtrivlist\@endpefalse}
\makeatother

\newcommand{\whp}{w.h.p.}

\newcommand{\E}{\mathbb E}

\newcommand{\eps}{\varepsilon}

\newcommand{\AC}[1]{{\textcolor{blue}{\textsc{(AC: #1)}}}}
\newcommand{\MZ}[1]{{\textcolor{red}{\textsc{(MZ: #1)}}}}

\DeclareMathOperator{\wsat}{wsat}
\DeclareMathOperator{\exc}{exc}
\DeclareMathOperator{\con}{con}

\DeclareMathOperator{\diam}{diam}

\DeclareMathOperator{\dist}{dist}

\RequirePackage{color}\definecolor{RED}{rgb}{1,0,0}\definecolor{BLUE}{rgb}{0,0,1} 

\begin{document}
\date{}
\maketitle

\begin{abstract}
Let $F$ and $G$ be two graphs. A spanning subgraph $H$ of $G$ is called weakly $F$-saturated if one can add to $H$ the edges of $G \setminus H$ in some order, so that whenever a new edge is added, a new copy of $F$ is formed. Obtaining lower bounds for the minimum size $\mathrm{wsat}(G,F)$ of such an $H$ is a classical problem in extremal combinatorics. In particular, in the past 40 years, various algebraic tools have been developed to prove lower bounds on the weak saturation number $\mathrm{wsat}(G,F)$. Our paper uncovers a new connection of weak saturation to topology of clique complexes, that allows to prove tight lower bounds in some cases when the algebraic tools are not efficient. 

It is easy to see that the smallest $K_3$-saturating graphs in $K_n$ are trees, thus $\mathrm{wsat}(K_n,K_3)=n-1$. In 2017, Kor\'andi and Sudakov proved that this is also the case in dense random graphs $G\sim G_{n,p}$, $p=\mathrm{const}\in(0,1)$, and posed the question of determining the smallest $p$ for which $G_{n,p}$ contains a $K_3$-saturating tree with high probability. 
 Using the new topological connection, we show that this critical $p$ is of order $n^{-1/3-o(1)}$. 

Inspired by Gromov's local-to-global principle for hyperbolic groups, we further develop our topological approach and determine the critical probability up to a constant factor, for trees with diameter at most $n^{c}$, for some $c>0$.


The new connection also enables us to improve the best known upper bound on the threshold probability for simple connectivity of the 2-dimensional clique complex of $G_{n,p}$, due to Kahle.

\end{abstract}

\section{Introduction}


Cellular automata were introduced by von Neumann~\cite{Neumann} after a suggestion of Ulam~\cite{Ulam}. The particular case, where the states of cells do not decrease, is known as {\it monotone} cellular automata or bootstrap percolation~\cite{Morris-automata}\footnote{More commonly, the notion of bootstrap percolation is referred to a monotone process with initially active cells chosen at random.}. 
 Although initially cellular automata were considered on lattices, since the foundational paper of  Bollob\'{a}s \cite{bol}, bootstrap percolation in arbitrary graphs attracted significant interest in extremal and probabilistic combinatorial communities (see,~e.g.,~\cite[Chapters 10, 11]{Survey}). In this paper, we study the {\it graph bootstrap percolation} suggested by Bollob\'{a}s~\cite{bol}, which is a substantial generalisation of $r$-neighborhood bootstrap percolation model having applications in physics; see,~\cite{Adler, Fontes,Morris}. Given a graph $F$, an {\it $F$-bootstrap percolation process} is a sequence of graphs $H_0\subset H_1\subset\cdots\subset H_m$ 
 such that, for  $i=1,\ldots,m$, the graph $H_i$ is obtained from $H_{i-1}$ by adding an edge that belongs to a copy of $F$ in $H_i$. A graph $H$ is {\it weakly $F$-saturated} in a graph $G$, if there exists an $F$-bootstrap percolation process $H=H_0\subset H_1\subset\cdots\subset H_m=G$. The minimum number of edges in a weakly $F$-saturated graph in $G$ is called the {\it weak $F$-saturation number} of $G$ and is denoted by $\mathrm{wsat}(G,F)$.

Despite a vast amount of research devoted to the graph bootstrap percolation~\cite{Alon85,AscoliHe,BBMR,BTT,EFT,Faudree,Frankl82,Kalai85,Kalai84,KMM,Morris-automata,MN,MNS,MShapira,Pikhurko-Phd,Pikhurko,ST-PAMS,Terekhov,TerekhZ-combi,TerekhZ,Tuza}, the exact value --- or even the asymptotics --- of $\mathrm{wsat}(G,F)$ is known only for some specific pairs of graphs $(G,F)$ and the complete picture is far from being well-understood. Even in the most well-studied case $G=K_n$, we are only aware of the following fairly general results describing possible values of  $\mathrm{wsat}(K_n,F)$: tight bounds on the weak saturation number for arbitrary $F$ and for $F$ from certain families~\cite{Faudree,TerekhZ-combi}; the existence of a non-trivial limit $c_F:=\lim_{n\to\infty}\mathrm{wsat}(K_n,F)/n$ for all $F$~\cite{Alon85}; a description of possible values of $c_F$~\cite{AscoliHe,TerekhZ-combi}.

One of the most efficient approaches to estimate from below the weak saturation number is linear algebraic, which in its most general form --- in terms of matroids --- was formulated by Kalai in~\cite{Kalai85}.  This approach allows to obtain the exact value of $\mathrm{wsat}(G,F)$, in particular, when both $G$ and $F$ are cliques~\cite{Alon85,Frankl82,Kalai85,Kalai84} and when $G$ is a clique and $F$ is a complete balanced bipartite graph~\cite{Kalai85,KMM}. It also allows to obtain asymptotically tight bounds when both $G$ and $F$ are complete bipartite graphs~\cite{KMM} and when $G$ and $F$ are grids~\cite{MNS}. Nevertheless, Terekhov and the last author of this paper proved~\cite{TerekhZ} that there are infinitely many $F$ such that the algebraic approach does not allow to find even the asymptotics of $\mathrm{wsat}(G=K_n,F)$, and the same is true for (asymptotically) almost all $G$.
In this paper, we supply a new topological approach to tackle weak saturation problems in a random setting. 
 Using the linear algebraic approach of \cite{BBMR} seems intractable in this case, and even known to fail when the ambient field is $\mathbb{F}_2$. In particular, it was shown that the cycle space and the triangle space coincide for `dense' values of $p$, see \cite{BarKah2019,DemHamKah2012,DubKah2024}. For these `dense' values of $p$ \whp\ $G_{n,p}$ is connected, and hence, the dimension of its cycle space is $e(G_{n,p})-n+1$. This implies that the lower bound given by the linear algebraic approach (with ambient field $\mathbb{F}_2$) is $n-1$, which also holds trivially by connectivity.

\paragraph{Random graphs.}  Over the past few decades, there has been considerable interest in extending extremal combinatorial results to random graph settings. A central focus of this research is to determine the threshold $p$ for which an extremal problem has the same solution on the complete graph and on $G_{n,p}$ \whp\footnote{With high probability, that is, with probability tending to 1 as $n\to\infty$.}  For example, Mantel's theorem states that the largest triangle free subgraph
of the complete graph contains $1/2$ of its edges (namely $\lfloor n^2/4\rfloor$) and is bipartite. The respective probabilistic question is to determine $p$ for which the largest triangle free subgraph of $G_{n,p}$ is bipartite. This question and its natural generalisation to graphs other than the triangle, known as Tur\'{a}n's problem, has been the focus of extensive research~\cite{BPS,CG-Annals,MK-Mantel,HS,HSZ,KKM,Schacht}. In the setting of weak saturation, it is thus natural to try to determine $p$, for which \whp
\begin{equation}
 \mathrm{wsat}(\mathbf{G},K_s)=\mathrm{wsat}(K_n,K_s)={n\choose 2}-{n-s+2\choose 2},\quad\text{ where }\mathbf{G}\sim G_{n,p}.
\label{eq:stability}
\end{equation}
This question was raised in 2017 by Kor\'{a}ndi and Sudakov who proved that~\eqref{eq:stability} holds \whp\ for every $s \geq 3$ and $p \in (0,1)$~\cite{KorSud}. 
They noted that their proof strategy can be used to show~\eqref{eq:stability} \whp\ when $p \geq n^{-\varepsilon}$ for
some sufficiently small constant $\varepsilon=\varepsilon(s)$ and raised the following question:
\begin{problem}
\label{Problem1}
For every fixed $s$, determine the range of $p$ for which~\eqref{eq:stability} holds \whp 
\end{problem}
This question appears to be complex even in the case when $s=3$. Observe that the equality 
$$
    \mathrm{wsat}(K_n,K_3)=n-1
$$
is straightforward for all $n\geq 3$ since every spanning subtree of a clique saturates its edges and since any weakly saturated subgraph has to be connected. In~\cite{Bidgoli}, it was proved that the stability property~(\ref{eq:stability}) exhibits a threshold probability $p_{K_s}$ and, in particular, for $s=3$,
\begin{equation}
 \left(\frac{3\ln n}{2n}\right)^{1/2}<p_{K_3}< \left(\frac{(\ln n)^2}{n}\right)^{1/3}.
\label{eq:old_bounds}
\end{equation}
More precisely, it was shown that if $p\geq \left(\frac{(\ln n)^2}{n}\right)^{1/3}$, then~\eqref{eq:stability} holds \whp, and, if $p\leq \left(\frac{3\ln n}{2n}\right)^{1/2}$, then~\eqref{eq:stability} does not hold \whp\  In this paper, we prove that $p_{K_3}=n^{-1/3-o(1)}$ --- in particular, we improve both the upper and the lower bound in~\eqref{eq:old_bounds} --- and give sharp (up to a constant factor) bounds on the threshold when the diameter of a weakly saturated tree is at most $n^{1/18-\varepsilon}$.

 
 In~\cite{Kal_Zhuk}, the result of Kor\'{a}ndi and Sudakov was extended to complete bipartite graphs and it was conjectured that~\eqref{eq:stability} takes place \whp\ for all graphs $F$. Although there is little hope to get explicit tight estimates of $\mathrm{wsat}(\mathbf{G},F)$ for all $F$, in~\cite{KMMT-R}, it was proved that, for any $F$, the weak saturation number is {\it asymptotically} stable: for any constant $p\in(0,1)$, \whp\
 $\mathrm{wsat}(\mathbf{G},F)=(1+o(1))\mathrm{wsat}(K_n,F)$.\\


The rest of Introduction is organised as follows. In Section~\ref{sc:results}, we state the main results of our paper. Then, in Section~\ref{sc:methods}, we explain the novel topological technique that we develop to prove the new lower bound. We show a connection between contractibility of loops in the triangle complex $X^{(2)}(G_{n,p})$ of $G_{n,p}$ and the existence of a weakly saturated subtree. In particular, our approach to prove lower bounds relies on the notion of hyperbolicity and the local-to-global principle of Gromov. Somewhat surprisingly, the bridge between topology and weak saturation also allows to derive a new bound on the threshold for simple connectivity of $X^{(2)}(G_{n,p})$ from our upper bound on $p_{K_3}$. Then, in Section~\ref{sc:poluted}, we discuss some other related work on bootstrap percolation in a, so called, polluted environment.




\subsection{New results}
\label{sc:results}

Let $H:=H_0\subset H_1\subset\cdots\subset H_s=:G$ be a $K_3$-bootstrap percolation process. We say that $H$ {\it activates} $G$ and write $H\to G$.

\begin{theorem}
\label{th:general}
     Let $\eps>0$, $p:=p(n)\in[0,1]$, and $\mathbf{G}\sim G_{n,p}$.
     \begin{enumerate}
     \item If $p\geq (1+\eps)n^{-1/3}$, then w.h.p.\ there exists a tree $T\subseteq \mathbf{G}$ such that $T\to \mathbf{G}$.
     \item If $p<n^{-1/3-\varepsilon}$, then w.h.p.\ there is no tree $T\subseteq \mathbf{G}$ such that $T\to \mathbf{G}$.
     \end{enumerate}
In particular, we have
$$
p_{K_3}=n^{-1/3-o(1)}.
$$ 
\end{theorem}

Theorem~\ref{th:general} thus determines the correct exponent of the threshold $p_{K_3}$ for the range in Problem~\ref{Problem1}. Note that both the upper and the lower bound in Theorem~\ref{th:general} improve~\eqref{eq:old_bounds}. 

The first part (1-statement) of Theorem~\ref{th:general} is proven constructively: we show that w.h.p.\ there exists a tree of {\it diameter 4} that activates $\mathbf{G}$.
On the one hand, it is easy to see that w.h.p.\ 4 is the minimum possible diameter of a subtree that activates $\mathbf{G}$. On the other hand, we do not know whether for $p=n^{-1/3-o(1)}$ there exists a tree of larger diameter that activates $\mathbf{G}$ w.h.p.\ and cannot be `reduced' to a tree of diameter 4 (see discussions in Section~\ref{sc:discussions}).
Below we state our second result showing that for such trees the upper bound on the threshold in Theorem~\ref{th:general} is tight. We manage to determine threshold probability (up to a constant factor) for all trees with diameter at most $n^{1/18-\varepsilon}$.


\begin{theorem}
\label{th:shallow}
     Let $\eps>0$, $p=p(n)\in[0,1]$, and $\mathbf{G}\sim G_{n,p}$.
     \begin{enumerate}
     \item If $p\geq (1+\eps)n^{-1/3}$, then w.h.p.\ there exists a tree $T\subseteq \mathbf{G}$ of diameter 4 such that $T\to \mathbf{G}$.
     \item If $p<(2^{-7/3}-\eps)n^{-1/3}$, then w.h.p.\ there is no tree $T\subseteq \mathbf{G}$ of diameter at most $n^{1/18-\eps}$ such that $T\to \mathbf{G}$.
     \end{enumerate}
\end{theorem}
Note that the first statement is just a refinement of the first statement of Theorem~\ref{th:general}, and that the second statement is not sufficient to determine the threshold for all trees --- in particular, \whp\ $\mathbf{G}$ has a Hamilton path that has diameter $n-1$ when $pn=\ln n+\ln\ln n+\omega(1)$~\cite{AKS,Ham3,Ham1,Ham2}. 

We derive the second part (0-statement) in Theorem~\ref{th:general} from the fact that, when $p<n^{-1/3-\varepsilon}$, w.h.p.\ the 2-dimensional clique complex $X^{(2)}(\mathbf{G})$ of $\mathbf{G}$ is not simply connected~\cite{Bab2012,CosFarHor2015}. Since the structure of the fundamental group of $X^{(2)}(\mathbf{G})$ is unknown for $p=n^{-1/3-o(1)}$, this connection to topology is not enough to further improve the lower bound on the activation threshold. The main novel conceptual and technical part of our paper is the new topological approach that investigates deeper the relation between the activation process and the structure of the fundamental group of $X^{(2)}(\mathbf{G})$. This new approach allows us to prove the 0-statement in Theorem~\ref{th:shallow}.
 We discuss the approach and give some flavour of the proof in the next section. 


\subsection{Methodology: simple connectivity of the triangle complex of the random graph}
\label{sc:methods}

Let $\mathbf{G}\sim G_{n,p}$ and let $\mathbf{X}:=X^{(2)}(\mathbf{G})$ be the 2-dimensional clique complex of $\mathbf{G}$. Our first crucial observation relating activation process and fundamental groups is that if there exists a spanning subtree $T\subseteq\mathbf{G}$ such that $T\to\mathbf{G}$, then $\mathbf{X}$ is simply connected.
\begin{lemma}\label{obs:activation implies simple connectivity}
    If there exists a spanning subtree $T\subseteq\mathbf{G}$ activating $\mathbf{G}$, then $\mathbf{X}$ is simply connected.
\end{lemma}
The proof is based on a connection between a $K_3$-bootstrap percolation process and homotopy equivalence: for every cycle $C\subseteq\mathbf{G}$, consider an activation process of $C$ from $T$ in $\mathbf{G}$. Then, follow this process backwards, edge by edge, starting from the last edge in $C$ that has been activated. At each step, we replace the activated edge with the two edges of the triangle that have been used to activate it. Eventually, $C$ is contracted to a closed walk along the tree $T$, 
 see details in Section~\ref{sc:general-0_proof}. 

It is known that the threshold for simple connectivity of $\mathbf{X}$ is $n^{-1/3-o(1)}$.

\begin{theorem}[Babson~\cite{Bab2012}, Costa, Farber, and Horak~\cite{CosFarHor2015}]
\label{th:Babson}
    Let $\eps>0$ be a constant, $p\leq n^{-1/3-\eps}$, $\mathbf{G}\sim G_{n,p}$, and $\mathbf{X}=X^{(2)}(\mathbf{G})$. Then, \whp\ $\mathbf{X}$ is not simply connected.
\end{theorem}


Then by combining Lemma~\ref{obs:activation implies simple connectivity} and Theorem~\ref{th:Babson} we are able to derive the 0-statement in Theorem~\ref{th:general}.

In the other direction, Lemma~\ref{obs:activation implies simple connectivity} can be also used to improve the upper bound on the simple connectivity threshold. Indeed,
the best known upper bound on the threshold for simple connectivity equals $\left(\frac{3\ln n+O(1)}{n}\right)^{1/3}$, due to Kahle~\cite{Kahle}. 
 Given Lemma~\ref{obs:activation implies simple connectivity}, the 1-statement in Theorem~\ref{th:general} immediately implies the following improvement. 

\begin{corollary}
Let $\eps>0$ be a constant, $p\geq(1+\eps)n^{-1/3}$, $\mathbf{G}\sim G_{n,p}$, and $\mathbf{X}=X^{(2)}(\mathbf{G})$. Then, \whp\ $\mathbf{X}$ is simply connected.
\end{corollary}

We prove the 1-statement in Theorem~\ref{th:general} constructively. We actually prove the 1-statement of Theorem~\ref{th:shallow} which is a stronger version of the former. In~\cite{Bidgoli}, it is proven that \whp\ the BFS tree $T$ rooted in any fixed vertex $v$ activates $\mathbf{G}$ when $p\geq((\ln n)^2/n)^{1/3}$. We note that w.h.p. this tree has depth 2.
The main reason why the BFS tree is sufficient is that, for such $p$, \whp\ for any vertex $u$ the common neighbourhood of $u$ and $v$ is connected. It is possible to push down the upper bound on the threshold to $(1+o(1))n^{-1/3}$ by observing that for $p>(1+\eps)n^{-1/3}$ \whp\ for every $u$, the common neighbourhood of $u$ and $v$ has a `giant' connected component of size $\Theta(np^2)$. This allows to choose the tree $T$ that activates $\mathbf{G}$ in a more delicate way, by attaching every vertex $u$ in the bottom layer of the tree to a vertex $u'$ of the middle layer so that $u'$ is in the giant component of the common neighbourhood of $u$ and $v$.

\paragraph{From local to global: proof overview of Theorem~\ref{th:shallow} part 2.}
The main novel part of our paper is the new technique that we develop to prove the 0-statement in Theorem~\ref{th:shallow}. Let us first observe that one of the standard tools in the theory of bootstrap percolation --- Aizenman-Lebovitz type property~\cite{AizenmanL} --- is not useful in our case. Recall that this property, roughly, asserts that there exists some constant $k$ such that, if some edge is activated via a graph on $m$ edges, then for every $m_0\leq m$, there exists an edge that uses at least $m_0$ and at most $km_0$ edges for its activation. Then, the usual way to apply this property is as follows (see, e.g.,~\cite{BBR,BrettZolta} for its applications in random graphs): First, prove that there exists some interval $I=I(n)=[i^-,i^+\geq k i^-]$ so that \whp\ there are no edges that can be activated via a graph whose size belongs to $I$. Then show that typically an edge requires a graph of size at least $i^-$ to activate it. Finally, apply the Aizenman-Lebovitz property to get that if some edge requires more than $i^+$ other edges for activation, then there exists another edge that is activated via a graph whose size belongs to $I$. In our case, this approach fails since, for any $i$ and every edge $e=\{v_1,v_2\}$, \whp\ there exists a subgraph of $G_{n,p}$ of size $i$ that activates $e$. Indeed, it is sufficient to find a sequence of vertices $v_3,v_4,\ldots$ such that each $v_j$ is adjacent to $v_{j-2}$ and $v_{j-1}$.




Let $X$ be a finite 2-dimensional path-connected simplicial complex and let $X^{(1)}$ be its 1-dimensional skeleton. The fundamental group $\pi_1(X)$ is called {\it Gromov--hyperbolic}, if for any Cayley graph of $\pi_1(X)$, equipped with the graph metric, there exists $\delta>0$ such that all geodesic triangles are $\delta$-thin~\cite{Gromov}.
This definition has an equivalent reformulation in terms of an isoperimetric inequality for loops~\cite{Gromov}:
 the fundamental group $\pi_1(X)$ is Gromov--hyperbolic, if and only if there exists a constant $C>0$ such that for every null-homotopic loop $\gamma:S^1\to X^{(1)}$ the area of $\gamma$ in its simplicial Van Kampen diagram admits a linear bound w.r.t.\  the length of the loop. In other words, $A_X(\gamma)\leq C|\gamma|$, where the area $A_X(\gamma)$  is the minimum number of triangles in a Van--Kampen diagram for $\gamma$, see the definition of a Van--Kampen diagram in Section~\ref{sc:pre-topology}. The proof of Theorem~\ref{th:Babson} in~\cite{Bab2012,CosFarHor2015} follows the blueprint of the proof of an analogous result in Linial--Meshulam complexes in~\cite{BabHofKah2011}, which essentially relies on the fact that $\pi_1(X^{(2)}(\mathbf{G}))$ is \whp\ hyperbolic. In turn, the proof of hyperbolicity of $\pi_1(X_2(\mathbf{G}))$ uses the {\it local-to-global principle} of Gromov~\cite{BabHofKah2011,Gromov}. This principle, roughly speaking, asserts that, if any pure subcomplex $Y\subset X$ with at most $C$ triangles is hyperbolic with the isoperimetric constant $\inf_{\gamma:S^1\to Y^{(1)}}\frac{|\gamma|}{A_Y(\gamma)}\geq\delta=\delta(C)$, then $X$ is hyperbolic.

Inspired by the Gromov's from local-to-global principle, we develop a local-to-global principle for activation processes. Omitting technical details, we show the following. Let us say that the activation process of a cycle $C$ is {\it contractible}, if $C$ is contractible in the complex of triangles participating in the process.
 Let $G$ be a graph. Fix positive integers $d$ and $\ell\geq 100d$, that might depend on $|V(G)|$. 
Suppose there exists a positive number $M\leq|V(G)|$ such that for any subgraph $H\subseteq G$ on at most $M$ vertices and any cycle $C\subseteq H$ of length at most $\ell$ which can be activated in $H$, it suffices to use only $M/100$ vertices to activate $C$ in $H$, via a contractible activation process.
Then, $M/100$ vertices also suffice to activate any cycle $C\subseteq G$ of length at most $\ell$ from a tree {\it of diameter at most $d$}.

 It then remains to show that, for an appropriate choice of $M=M(n)$ w.h.p.
\begin{enumerate}[label=\textbf{(P\arabic*)}]
        \item\label{ppty1} there is a cycle of length at most $100d$ that cannot be activated by a process that uses at most $M/100$ vertices;
        \item\label{ppty2} no cycle of length at most $100d$ can be activated by a process that {\it essentially} exploits more than $M/100$ but at most $M$ vertices. In other words, if a cycle can be activated via a contractible process that exploits at most $M$ vertices, then it can also be activated by a contractible process with at most $M/100$ vertices.
\end{enumerate}
In order to prove both properties of $\mathbf{G}$, let us first observe that it is easy to count `planar' activation processes. Indeed, consider a triangulation $D$ of a cycle $C$ and its spanning subtree $T\subset D$. Then, clearly, $T$ activates $C$ in $D$. Having enumeration results for triangulations of a cycle $C$ of a given length and with a fixed number of internal vertices~\cite{BerFus2012,Tut1962}, one can estimate the expected number of such planar triangulations of a fixed cycle in $\mathbf{G}\sim G_{n,p}$ and derive that w.h.p.\ typical cycles in $\mathbf{G}$ cannot be activated in a `planar' way when $p\leq\left(3/4^{4/3}-\varepsilon\right)n^{-1/3}$. Even though we do believe that these `planar' processes are the cheapest way to activate a cycle, we do not manage to prove that. Nevertheless, we show that locally this is not far from being the truth, confirming the above two properties when $p\leq (2^{-7/3}-\eps)n^{-1/3}$. Instead of using the enumeration results for planar triangulations, we bound from above the number of graphs whose edges participate in an activation process of a given cycle by introducing the notion of an {\it activation diagram}, which resembles the notion of Van Kampen diagram with the key difference that the activation diagram contains each triangle from $X^{(2)}(\mathbf{G})$ at most once and it is not simply connected.
We reduce the problem of counting these diagrams to the problem of counting pairs of trees and certain walks in them. Our upper bound on the number of diagrams is sufficient to derive both properties \ref{ppty1} and \ref{ppty2}.  


Unfortunately, when the length of a cycle exceeds $n^{\alpha}$ for a certain $\alpha\in(0,1)$\footnote{In the current proof $\alpha=1/18$, but 
 we did not try to optimise this constant, as there is a critical value $\alpha=1/3$, that our methods do not allow to exceed.}, our bounds on the number of activation diagrams become ineffective. This barrier does not allow us to generalise the 0-statement in Theorem~\ref{th:shallow} for all trees. We suspect that deriving a stronger version of our from-local-to-global principle that does not restrict the diameter of a tree seems more promising than improving our enumeration results for activation processes.

\begin{remark}
As we already mentioned, the Aizenman--Lebovitz property \cite{AizenmanL} cannot be used to implement our proof strategy, although its assertion is similar to our main technical lemmas, Lemma \ref{lem: no ell-cycle between log and Klog} and Lemma \ref{lem: exists an ell-cycle with more than log}.
Indeed, Aizenman--Lebovitz type arguments find an element with restricted `activation speed' --- similarly, we find such a `short' cycle. Crucially, a cycle is built from a few edges, and Aizenman-Lebovitz arguments only yield the existence of single edges with restricted `activation speed'.  
\end{remark}


\subsection{Related work: Bootstrap percolation in a polluted environment}
\label{sc:poluted}

Let us mention another line of research related to our results. In this paper, we study a model of bootstrap percolation where the domain set is random. Bootstrap percolation in random domains is also referred to as {\it bootstrap percolation in a polluted environment}. 
 A model of polluted bootstrap percolation on lattices was introduced by Gravner and McDonald~\cite{GravMcD} and has been actively studied in recent years, see, e.g.,~\cite{GravHolroyd,GHS}. Furthermore, Gravner and Kolesnik~\cite{GravKolesnik} studied such processes in `polluted cliques', i.e. in $G(n,p)$. In our settings, their result can be stated as follows.
\begin{theorem}[Gravner, Kolesnik~\cite{GravKolesnik}]\label{GravKole}
Let $D>0$ be a fixed constant and let $H=H(n)$ be a fixed sequence of graphs on $[n]$ with maximum degrees at most $D$. Let $\mathbf{G}\sim G_{n,p}$ be sampled on the same vertex set $[n]$. There exist positive constants $c=c(D)$ and $C=C(D)$ such that
\begin{enumerate}
\item if $p<\frac{c}{\sqrt{\ln n}}$, then \whp\ $H\not\to H\cup\mathbf{G}$.
\item if $p>\frac{C\ln\ln n}{\sqrt{\ln n}}$, then \whp\ $H\to H\cup\mathbf{G}$.
\end{enumerate}
\end{theorem}
 Actually, the tree of diameter 4 that \whp\ activates $\mathbf{G}\sim G_{n,p}$ when $p>(1+\varepsilon)n^{-1/3}$ can be reduced to a weakly saturated Hamilton path, see details in Section~\ref{sc:discussions}. Therefore, for bounded degree trees the threshold probability for the {\it existence} of a weakly saturated subtree from Theorem~\ref{th:general} is much lower than the threshold probability for the property that a {\it fixed} tree activates the random graph\footnote{By applying the methods from~\cite{Bidgoli} verbatim, one can also demonstrate that this property has a threshold probability.}, given by Theorem~\ref{GravKole}.


In order to explain the reason for this difference for paths, assume that a fixed path $P$ activates at least one edge of $\mathbf{G}\sim G_{n,p}$ between two vertices at distance $\Theta(\ln n)$ in $P$. Then, let $\{u,v\}$ be the first such edge in the respective activation process. Take the minimum subpath $P'\subset P$ that activates $\{u,v\}$. In the respective subprocess all edges $\{x,y\}$ are {\it short}, i.e. $d_P(x,y)<d_P(u,v)$. From this one can derive that there are $\Theta(|V(P')|)$ triangles $\{x,y,z\}$ such that either $\{x,y\}\in E(P')$ or $\{x,y\}\in\mathbf{G}$ has length 2 in $P'$, and $\{x,z\},\{y,z\}$ are short edges of $\mathbf{G}$. The latter event is not likely when $p<\frac{c}{\sqrt{\ln n}}$, for a sufficiently small $c>0$ (for more details see~\cite[Section 2]{GravKolesnik}).

On the other hand, when we are able to choose the initially active subtree of $\mathbf{G}$, it is clearly possible to activate long edges for $p$ that is even much less than the threshold from Theorem~\ref{th:general}. Indeed, fix any vertex $v$, expose its neighbourhood $N(v)$ in $\mathbf{G}$, and take a Hamilton path $P$ in $N(v)$. Such a path exists \whp\ when $p\gg\sqrt{\ln n/n}$. Extend $P$ by attaching the vertex $v$ to one of its ends. This path activates all edges that contain $v$. The longest such edge has length $(1+o(1))np$ in~$P$ \whp

\subsection{Organisation}
Section~\ref{sc:pre} contains the required prerequisites; it is separated into two subsections --- Section~\ref{sc:pre-topology} recalls topological concepts that we use in our proofs and Section~\ref{sc:pre-RG} contains several useful properties of random graphs. In Section~\ref{sc:proof_shallow_1} we prove the 1-statement from Theorem~\ref{th:shallow}, which, in turn, immediately implies the 1-statement in Theorem~\ref{th:general}. In Section~\ref{sc:general-0_proof}, we prove Lemma~\ref{obs:activation implies simple connectivity}, which is reiterated in the deterministic setting as Lemma~\ref{lem: main lemma for 0-statement}, and then combine it with Theorem~\ref{th:Babson} to derive the 0-statement in Theorem~\ref{th:general}. 
 The main novel part of our paper --- the proof of the 0-statement in Theorem~\ref{th:shallow} --- is presented in Section~\ref{sec:main proof}. 
  Finally, Section~\ref{sc:discussions} is devoted to discussions of our results, their generalisations, and further challenges. In particular, we give a sketch of the proof of an analogue of Theorem~\ref{th:shallow} for Linial--Meshulam random 2-complexes. 

\section{Preliminaries}
\label{sc:pre}

\subsection{Simplicial complexes and Van Kampen diagrams}
\label{sc:pre-topology}

Throughout the paper, it will be convenient to treat simplicial complexes and the homotopy of loops as combinatorial objects. For the sake of completeness, we recall the respective definitions in this section.



Let $V$ be a finite set. An {\it (abstract) simplicial complex} $X$ is a downwards-closed family of subsets of $V$. Every element $S\in X$ is a {\it face}, and its {\it dimension} is $|S|-1$. The 0-dimensional faces are called {\it vertices}. We denote the set of vertices of $X$ by $V(X)$. For a positive integer $k$, we say that a simplicial complex $X$ is {\it $k$-dimensional} if the maximum dimension of a face in $X$ is $k$. Moreover, the {\it $k$-skeleton of $X$} is the simplicial complex $X^{(k)}$ consisting of all faces of $X$ that have dimensions at most $k$.


\begin{definition}[Simplicial Map]
    Let $X,Y$ be simplicial complexes. A function $\varphi:V(X)\to V(Y)$ is called a \emph{simplicial map}, if, for every $S\in X$, we have $\varphi(S)\in Y$. Moreover, we say that a simplicial map $\varphi$ is an \emph{isomorphism} if $\varphi$ is a bijection, and $\varphi^{-1}$ is a simplicial map as well. 
\end{definition}

In this paper, we consider 2-dimensional complexes. We call the 1-dimensional faces of a complex $X$ {\it edges}. We will denote by $E(X)$ and $F(X)$ the sets of edges and 2-dimensional faces of $X$, respectively.

 Let $D$ be a {\it triangulated disc}, i.e. a 2-dimensional complex (not necessarily pure, i.e. some edges may not belong to triangles) comprising faces of a plane embedding of a planar graph, that has only triangular faces, excluding the `external face'\footnote{In topological terms, there exists a simplicial map $\gamma:S^1\to D$ such that the mapping cylinder of $\gamma$ is a disc with boundary $S^1\times\{0\}$.}. We denote by $\partial D$ the boundary of $D$.


\begin{definition}[Contractible Cycle]\label{def:contractible_cycle}
    Let $X$ be a simplicial complex, and let $C\subseteq X^{(1)}$ be a cycle. We say that a triangulated disc $D$ is a \emph{simplicial filling / Van Kampen  diagram of $C$ in $X$}, if there exists a simplicial map $\varphi:V(D)\to V(X)$ such that $\varphi|_{V(\partial D)}$ is an isomorphism between $\partial D$  and $C$. Moreover, if there exists a simplicial filling of $C$ in $X$, we say that $C$ is \emph{contractible} in~$X$.
\end{definition}

Notice that a simplicial filling of a cycle is necessarily pure. However, we can extend Definition~\ref{def:contractible_cycle} to closed walks, and then a simplicial filling may not be pure. If $C$ is a closed walk in $X^{(1)}$, then its {\it simplicial filling} $D$ --- which is still a 2-dimensional complex embedded into the plane but not necessarily homeomorphic to a disc --- and the respective simplicial map $\varphi$ satisfy the following property: there exists a closed walk along the entire boundary that visits the first vertex exactly twice and each edge at most twice, such that the $i$-th edge of this walk maps to the $i$-th edge of $C$\footnote{The topological definition involving mapping cylinders remain unchanged, since the mapping cylinder is homeomorphic to a closed disc in either case.}, see Figure \ref{Fig-VK-diagram-walks}. More formally, let $D$ be any planar connected 2-dimensional simplicial complex such that all the bounded complementary regions for the graph $D^{(1)}$ are 2-dimensional faces of $D$. For a fixed vertex $x$ on the boundary of $D$, there exists a unique (up to the choice of the direction) walk $U:=U(D,x)=(u_1,\ldots ,u_t)$ along the boundary such that $u_1=u_t=x$, and every edge of the boundary is visited at most twice by the walk $U$. 
 The planar complex $D$ is a {\it simplicial filling} of the walk $W=(w_1,\ldots,w_t)$, if there exists a simplicial map $\varphi : V(D)\to V(W)$ such that, for some choice of $x\in V(\partial D)$ and for $(u_1,\ldots,u_t)=U(D,x)$, we have that $\varphi(u_1)=w_1$ and, for every $i\in[t-1]$, $\varphi(u_iu_{i+1})=w_iw_{i+1}$. 

\begin{figure}[ht]
\centering
    \includegraphics[scale=0.49]{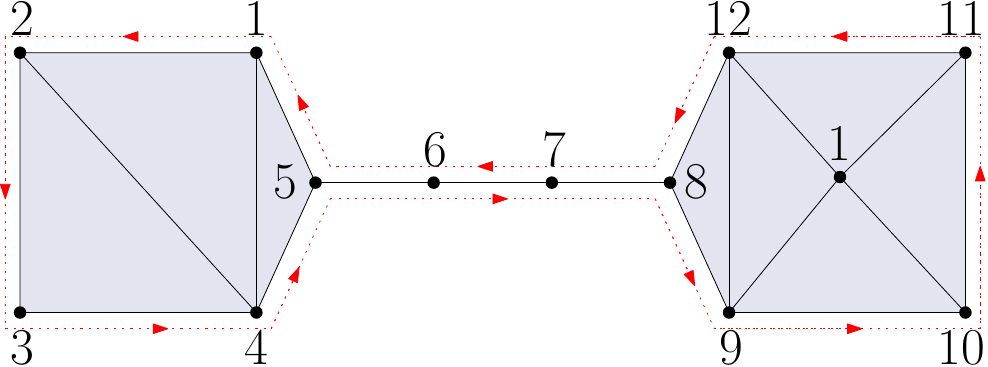}
    \caption{A Van-Kampen diagram for a closed walk, presented in red.}\label{Fig-VK-diagram-walks}
\end{figure}




\begin{definition}[Simply Connected Complex]
    A simplicial complex $X$ with a connected skeleton $X^{(1)}$ is {\it simply connected}, if every cycle $C\subseteq X^{(1)}$ is contractible in $X$ (or, equivalently, every closed walk is contractible).  
\end{definition}

Let $G$ be a graph. The {\it clique complex} of $G$ is the simplicial complex $X(G)$ that consists of all cliques of $G$. In particular $X^{(2)}(G)$ comprises all triangles, edges, and vertices of $G$.

\subsection{Random graphs}
\label{sc:pre-RG}

In this section, we describe certain properties of $\mathbf{G}\sim G_{n, p}$ that we will use in the proofs. First of all, we recall that degrees of vertices in $G_{n,p}$ have binomial distribution and, thus, they are well-concentrated.

\begin{claim}[\cite{Bollobas}, Chapter 3]
\label{cl:degrees}
If $pn\gg\ln n$, then w.h.p.\ every vertex has degree $(1+o(1))np$ in $\mathbf{G}$.
\end{claim}

The same applies to co-degrees. Indeed, for fixed different vertices $u,v\in[n]$, their common neighbourhood $N(u)\cap N(v)$ has size $\mathrm{Bin}(n-2,p)$. Therefore, due to the Chernoff bound (see, e.g.~\cite[Theorem 2.1]{Janson}), the following is true:
\begin{claim}
\label{cl:co-deg}
If $np^2\gg\ln n$, then 
$$
 \mathbb{P}\biggl(\forall u,v\quad |N(u)\cap N(v)|=(1+o(1))np^2\biggr)=1-\exp(-\omega(\ln n)).
$$
\end{claim}

We also recall that, if for some positive constant $\eps$, we have $np\geq 1+\eps$, then w.h.p.\ $\mathbf{G}$ contains a unique component of linear size in $n$. Moreover, the size of this component $\mathbf{g}_n$ has exponential tails.
\begin{theorem}[O'Connell \cite{Oco1998}]\label{th:giant}
Let $C=C(\varepsilon)$ satisfy $C=1-e^{-(1+\varepsilon)C}$ and suppose that $np\geq 1+\eps$. Then, for every $\delta>0$, with probability $1-\exp(-\Omega(n))$, the graph $\mathbf{G}$ contains a unique component of size $\mathbf{g}_n>(1-\delta)Cn$ and all other components have size less than $(1-2\delta)Cn$.
\end{theorem}

Finally, we need the following theorem about the threshold for the existence of extensions in random graphs. A rooted graph $(K,S)$, where $K$ is a graph and $S\subseteq V(K)$ is a proper {\it ordered} subset, is called {\it strictly balanced}, if, for every set $S'\supseteq S$, which is also a proper subset of $V(K)$,
$$
 \frac{|E(K[S'])|-|E(K[S])|}{|S'|-|S|}<
 \frac{|E(K)|-|E(K[S])|}{|V(K)|-|S|}=:\rho(K,S).
$$
The quantity $\rho(K,S)$ is called the {\it density} of the rooted graph $(K,S)$. Rooted graphs $(\tilde K,\tilde S)$ and $(K,S)$ are {\it isomorphic}, if there exists a bijection $\varphi:V(K)\to V(\tilde K)$ with the following properties: 
\begin{itemize}
\item for every $i\in[|S|]$, the bijection $\varphi$ maps the $i$-th vertex from $S$ to the $i$-th vertex from $\tilde S$;
\item for every pair of vertices $u,u'\in V(K)$, such that at least one of them is not in $S$, we have that $\{u,u'\}\in E(K)$ if and only if $\{\varphi(u),\varphi(u')\}\in E(\tilde K)$. 
\end{itemize}

\begin{theorem}[Spencer \cite{Spe1990}]
\label{th:extensions}
Let $(K,S)$ be a strictly balanced rooted graph with density $\rho$. There exists a constant $C=C(K,S)>0$ such that, if $p>(\ln n)^{C} n^{-1/\rho}$ and $\mathbf{G}\sim G_{n,p}$, then w.h.p., for every set $\tilde S\subseteq [n]$ of size $|S|$, there exists a subgraph $\tilde K\subseteq \mathbf{G}$ such that the rooted graphs $(\tilde K,\tilde S)$ and $(K,S)$ are isomorphic.
\end{theorem}

\section{Proof of Theorem~\ref{th:shallow}, part 1}
\label{sc:proof_shallow_1}


Let $\eps>0$ be a fixed constant, $p:=p(n)\geq(1+\eps)n^{-1/3}$, and $\mathbf{G}\sim G_{n,p}$. In this section, we construct a tree of depth 2 that activates $\mathbf{G}$ w.h.p. In~\cite{Bidgoli}, it is proven that, if we slightly increase the lower bound on $p$ --- namely, for $p\gg(\ln n/n)^{1/3}$ --- then, w.h.p. the BFS tree explored from any vertex has depth 2 and activates $\mathbf{G}$. Here we prove that a slightly more delicate way of attaching vertices at the second layer of the tree to the vertices at the first layer allows to activate $\mathbf{G}$ w.h.p.\ for the entire range $p>(1+\eps)n^{-1/3}$. 


In what follows, we denote by $N(v)$ the set of neighbours of $v$ in $\mathbf{G}$. Fix a sufficiently small constant $c \coloneqq c(\eps)>0$. Let $v\in[n]$ be an arbitrarily chosen vertex. 
For every $u\in [n]\setminus(N(v)\cup\{v\})$, let $H(u)$ be the graph induced by the common neighbourhood of $u$ and $v$, that is 
 $$
 H(u):=\mathbf{G}[N(u)\cap N(v)].
 $$
 Due to Claim~\ref{cl:co-deg} and Theorem~\ref{th:giant}, w.h.p., for every $u\in [n]\setminus(N(v)\cup\{v\})$, there exists a component $C(u)$ of size at least $cn^{1/3}$ in $H(u)$. Indeed, we first fix any vertex $u$ and expose all edges containing $v$. 
  If $u\in N(v)$, then there is nothing to check. Otherwise, expose all edges containing $u$. Due to Claim~\ref{cl:co-deg}, 
  $|V(H(u))|=(1+o(1))np^2$ with probability $1-o(1/n)$. By Theorem~\ref{th:giant} and since 
  $$
   |V(H(u))|p=(1+o(1))np^3>1+\varepsilon
  $$
    with probability $1-o(1/n)$, there exists a component $C(u)$ of size at least $cn^{1/3}$ in $H(u)$ with the same probability bound. 
     By the union bound over $u\in[n]$, we get the desired assertion.




The tree $T\subseteq \mathbf{G}$ such that $T\to \mathbf{G}$ w.h.p.\ is constructed as follows. First of all, it contains all edges between $v$ and $N(v)$. Second, for every vertex $u\notin N(v)\cup\{v\}$, add to $T$ an edge from $u$ to the giant component $C(u)\subseteq H(u)$. Let us show that w.h.p. $T\to \mathbf{G}$. 

Clearly, any edge $e$ of $\mathbf{G}[N(v)]$ can be activated since its addition completes the triangle $e\cup\{v\}$. Let us show that we may also activate all edges between $N(v)$ and $[n]\setminus (N(v)\cup\{v\})$. Let $u\in[n]\setminus(N(v)\cup\{v\})$. We activate the edges between $u$ and $H(u)$ in the following way. Let $u'$ be the neighbour of $u$ in $T$. The activation of edges between $u$ and $C(u)$ is straightforward: for every such edge $\{u,w\}$, there exists a path $u'=w_0w_1\ldots w_s=w$. Then, the edges $\{u,w_1\},\ldots,\{u,w_s\}$ can be activated sequentially --- the edge $\{u,w_i\}$ completes the triangle $\{u,w_{i-1},w_{i}\}$.

Now, let $w\in V(H(u))\setminus V(C(u)) $ and let $\tau=\lfloor n^{2/3}(\ln n)^2\rfloor$. Order the vertices of $[n]\setminus(N(v)\cup\{u,v\})$ in an arbitrary way: $u_1,\ldots,u_m$. Note that \whp\ $m=n-np(1+o(1))\gg\tau$. We will use the following assertion to show the activation procedure for $\{u,w\}$. 

\begin{claim}\label{cl:giant_to_giant}
With probability $1-o(1/n^2)$, there exist $i\in[\tau]$ and a vertex $z\in V(C(u))$ such that $\{u,u_i\}\in E(\mathbf{G})$ and $w,z\in V(C(u_i))$. 
\end{claim}

\begin{proof}
For every $i=1,\ldots,\tau$, we expose sequentially some sets of edges `associated' with $u_i$. These sets of edges are defined in the following way. Assume that, before  step $i\in[\tau]$, all edges between $\{u=:u_0,u_1,\ldots,u_{i-1}\}$ and $N(v)\cup\{u\}$ are exposed. Also, for every $j\in\{0,\ldots,i-1\}$, 
 the edges in a certain set $U_j\subseteq N(u_j)\cap N(v)$, that is defined below, are exposed. Set $U_0=V(H(u))$.

At step $i$, we first expose the adjacency relations between $u_i$ and $\{u\}\cup N(v)$. Let us consider the following event $\mathcal{E}_i$:
\begin{itemize}
\item $\{u_i,u\},\{u_i,w\}\in E(\mathbf{G})$;
\item $u_i$ has a neighbour in $C(u)$.
\end{itemize}
If $\mathcal{E}_i$ is not satisfied, set $U_i=\emptyset$. Otherwise, set
$$
U_i:=(N(u_i)\cap N(v))\setminus (U_0\sqcup U_1\sqcup\ldots\sqcup U_{i-1})
$$ 
and expose the edges of $\mathbf{G}[U_i]$. Also, let $z\in V(C(u))$ be a neighbour of $u_i$. Subject to the event that \begin{equation}
\label{eq:common_neighbourhood_exposed}
|N(v)\setminus(U_0\sqcup U_1\sqcup\ldots\sqcup U_{i-1})|=(1+o(1))np,
\end{equation}
the set $U_i$ satisfies
\begin{equation}
\label{eq:U_i_cardinality}
|U_i|=(1+o(1))np^2
\end{equation}
 with probability $1-o(n^{-3})$ by the Chernoff bound. Due to Theorem~\ref{th:giant}, with probability $1-o(n^{-3})$, there exists a component of size at least $cn^{1/3}$ in $\mathbf{G}[U_i]$. If such a component exists, then it has neighbours of $z$ and $w$ with probability at least 
\begin{equation}
\label{eq:two_neighbours}
\left(1-(1-p)^{cn^{1/3}}\right)^2\sim \left(1-e^{-(1+\varepsilon)c}\right)^2.
\end{equation}
As soon as we find such neighbours in $U_i$, we get the desired $i\in[\tau]$.

Let $i^*$ be the first $i\in[\tau]$ such that at least $(\ln n)^2$ events $\mathcal{E}_j$, $j\leq i$, hold, see Figure \ref{Fig1}.  If no such $i$ exists, set $i^*=\tau+1$. 
 Since the events $\mathcal{E}_i$, $i\in[\tau]$, are independent and 
$$
\mathbb{P}(\mathcal{E}_i)=p^2\left(1-(1-p)^{|C(u)|}\right)=\Theta(n^{-2/3})
$$
for all $i\in[\tau]$, we get that $i^*\leq \tau$ with probability $1-o(n^{-2})$. Due to Claim~\ref{cl:degrees}, with probability $1-o(n^{-2})$, 
$$
(1-o(1))np-(1+o(1))\ln^2 n\cdot np^2\leq |N(v)\setminus(U_0\sqcup U_1\sqcup\ldots\sqcup U_{i})|=(1+o(1))np
$$
for all $i\leq i^*$. 
Due to~\eqref{eq:two_neighbours} and Theorem~\ref{th:giant}, with probability $1-o(n^{-2})$, there exists $i\leq i^*$ such that $\mathcal{E}_i$,~\eqref{eq:common_neighbourhood_exposed},~and~\eqref{eq:U_i_cardinality} hold, the largest component in $U_i$ is inside $C(u_i)$, and this component has neighbours of $z$ and $w$, completing the proof.
\end{proof}

\begin{figure}[ht]
\centering
    \includegraphics[scale=0.49]{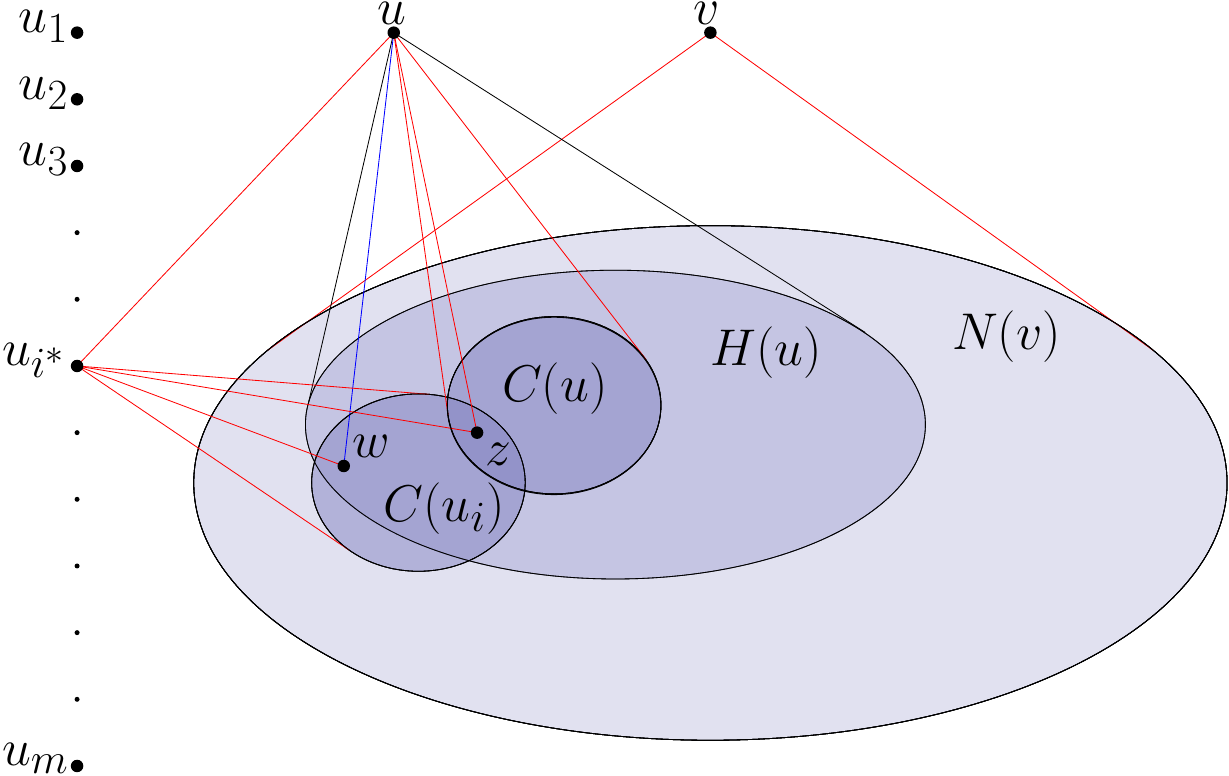}
    \caption{Activation of the blue edge $\{u,w\}$. Red edges participate in the activation process. Black edges show the entire neighbourhood of $u$ in $N(v)$.}\label{Fig1}
\end{figure}

Due to Claim~\ref{cl:giant_to_giant} and the union bound over all pairs of vertices from $N(v)\times([n]\setminus (N(v)\cup\{v\})$, w.h.p.\ we may activate all edges between $N(v)$ and $[n]\setminus (N(v)\cup\{v\})$. In particular, the edge $\{u,w\}$ is activated in the following way: first of all, we may assume that edges $\{u_i,z\},\{u_i,w\}$ are already activated since $z,w\in C(u_i)$. Then, the edge $\{u,u_i\}$ is activated via the triangle $\{u,z,u_i\}$. Finally, the desired edge $\{u,w\}$ is activated via the triangle $\{u,u_i,w\}$.


The fact that we may activate all the remaining edges follows immediately from the following claim.

\begin{claim}
\label{cl:extension}
W.h.p., for every pair of vertices $u,\tilde u\notin N(v)\cup\{v\}$, there exist vertices $w,\tilde w$ in $N(v)$ and a vertex $x\in[n]\setminus N(v)$ such that the pairs $\{u,w\}$, $\{\tilde u,\tilde w\}$, $\{w,x\}$, $\{\tilde w, x\}$, $\{u,x\}$, $\{\tilde u,x\}$ are adjacent in $\mathbf{G}$.
\end{claim}

\begin{proof}
The assertion follows from Theorem~\ref{th:extensions} and the fact that the rooted graph $(K,S)$, where 
\begin{itemize}
\item $V(K)=\{u,\tilde u,w,\tilde w,x,v\}$, \item $E(K)$ contains all the edges required by the claim as well as $\{v,w\},\{v,\tilde w\}$ (see Figure~\ref{fg:CL_3.2}), and 
\item $S=\{v,u,\tilde u\}$,
\end{itemize}
is strictly balanced and has density $8/3$.
\end{proof}

\begin{figure}[ht]
\centering
    \includegraphics[scale=0.49]{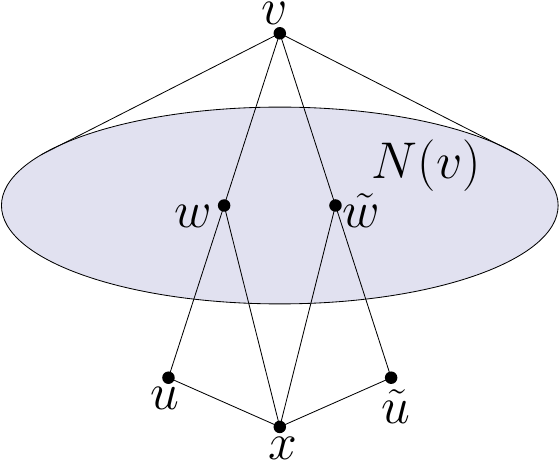}
    \caption{The graph present in $G_{n,p}$ due to Claim \ref{cl:extension}}\label{fg:CL_3.2}
\end{figure}

Let us complete the proof of the first part of Theorem~\ref{th:shallow}. Let $\{u,\tilde u\}$ be an edge of $\mathbf{G}$ in $[n]\setminus(N(v)\cup\{v\})$ that we want to activate and assume that the vertices $w,\tilde w,x$ from the assertion of Claim~\ref{cl:extension} exist. Then the edge $\{u,\tilde u\}$ can be activated in the following way: We may assume that the edges $\{u,w\},\{\tilde u,\tilde w\},\{x,w\},\{\tilde x,\tilde w\}$ are already activated. Then we may activate edges $\{u,x\}$ and $\{\tilde u,x\}$ via the triangles $\{u,x,w\}$ and $\{\tilde u,x,\tilde w\}$, respectively. These two edges, in turn, activate the desired $\{u,\tilde u\}$. 

\section{Proof of Lemma~\ref{obs:activation implies simple connectivity} and Theorem~\ref{th:general}, part 2}
\label{sc:general-0_proof}

Let $n^{-1/2}<p<n^{-1/3-\varepsilon}$, $\mathbf{G}\sim G_{n,p}$. The goal of this section is to prove that w.h.p.\ $\wsat(\mathbf{G},K_3) \neq n-1$ or, in other words, that w.h.p.\ there is no spanning tree that activates $\mathbf{G}$. It immediately follows from Theorem~\ref{th:Babson} and the following lemma.

\begin{lemma}
\label{lem: main lemma for 0-statement}
    If there exists a spanning tree $T$ in a graph $G$ such that $T\to G$, then $X^{(2)}(G)$ is simply connected.
\end{lemma}

Let us first give an `informal proof' of Lemma~\ref{lem: main lemma for 0-statement}. For any cycle $C\subseteq G$ consider the process of activation of its edges from $T$. Then, $C$ can be contracted in the following way: at step $i$, we have a closed walk $C_i$ on $G$ (in particular, $C_0=C$). Let $\{x_1,x_2\}\in C_i$ be the last activated edge in $C_i$ (it may have several instances in $C_i$) and let $\{x_1,x_2,x_3\}$ be the triangle that was used to activate $e$. Then replace each $\{x_1,x_2\}$ in $C_i$ with the pair of edges $\{x_1,x_3\},\{x_2,x_3\}$ that was used to activate it. Eventually, we get a walk on $T$ which is clearly contractible.

In order to formalise the above argument, we introduce an algorithm to construct a van Kampen diagram in the specific case when a cycle can be activated from a tree. We call the output of this algorithm the {\it algorithmic van Kampen diagram}. For a bootstrap percolation process $H:=H_0\subset H_1\subset\cdots\subset H_s=:G_0$ on edges of a graph $G$ initiated at a subgraph $H\subseteq G$ and a graph $G_0\subseteq G$ activated in this process, we will write $H\stackrel{G}\to G_0$. For every $i\in[s]$, let $\Delta_i\subseteq H_i$ be a triangle the edge $H_i\setminus H_{i-1}$ belongs to. We will call the tuple $(H;\Delta_1,\ldots,\Delta_s)$ an {\it activation process}.



\begin{definition}[Algorithmic Van Kampen Diagram]
    Suppose that $H\subseteq G$ are graphs and $C\subseteq G$ is a cycle such that $H$ activates $C$ in a process $\mathbf{A}=(H;\Delta_{1},\ldots ,\Delta_{s})$.
    Moreover, for every $1\leq k\leq s$ denote by $e_k=\{x_k^1,x_k^2\}$ the edge that was activated at step $k$, i.e.\ $e_k\coloneqq \Delta_{k}\setminus \left(H\cup \bigcup_{i=1}^{k-1}\Delta_i\right)$. Define the following sequence of labelled complexes iteratively (labels of different vertices are not necessarily different):
    \begin{enumerate}
        \item Start by initialising $D_{s+1}:=C$.
        \item Let $k\geq 1$ and suppose that $D_{k+1}$ is defined. Then, define the complex $D_{k}$ by gluing the triangle $\Delta_{k}=e_k\cup\{y_k\}$, to every edge of $D_{k+1}$ whose vertices are labelled by $x_k^1$ and $x_k^2$. For every attached instance of $\Delta_k$, label its third vertex by $y_k$.
    \end{enumerate}
    We call the resulting labelled simplicial complex $D_1$, the \emph{algorithmic $\mathbf{A}$-diagram} of $C$. Note that $D_1$ may have {\it different} vertices labelled by the same label. We also note that an embedding of $D_1$ into the plane is homotopic to an embedding of an annulus. The labelled version of the boundary of the `hole' contains only edges of $H$. See Figure \ref{Fig-diagrams} that compares different diagrams.
\end{definition}

\begin{proof}[Proof of Lemma~\ref{lem: main lemma for 0-statement}.]

    Suppose that $G$ can be activated from a tree $T\subseteq G$. Consider any cycle $C\subseteq G$ and its algorithmic van Kampen diagram $D=D(C)$. Since the preimage $W$ of the boundary of the hole of an embedding of $D$ into $\mathbb{R}^2$ consists of edges of $T$ only, the sequence of consecutive edges of $W$ is a closed walk along $T$. Then, all the edges of this walk can be glued via the following iterative algorithm: if $W$ contains a pair of sequential edges $\{v_i,v_{i+1}\}$ and $\{v_{i+1},v_i\}$, then identify these two edges. Eventually, we get a simplicial complex $D'$ that is a simplicial filling of $C$, completing the proof.
\end{proof}

\begin{figure}[H]
    \centering
    \includegraphics[scale=0.49]{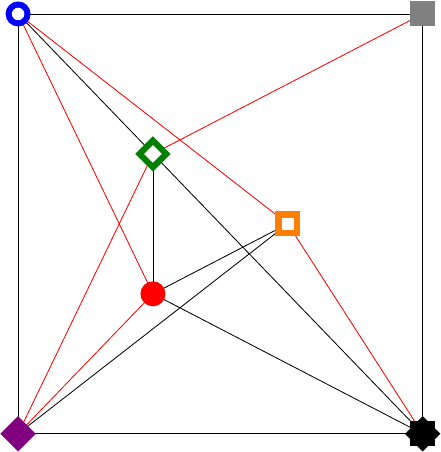}
    \caption*{The graph with an activation tree in red}
    \label{Figures/Fig-the-graph}

\vspace{0.45cm}

\begin{minipage}{0.5\textwidth}
  \centering
  \includegraphics[scale=0.49]{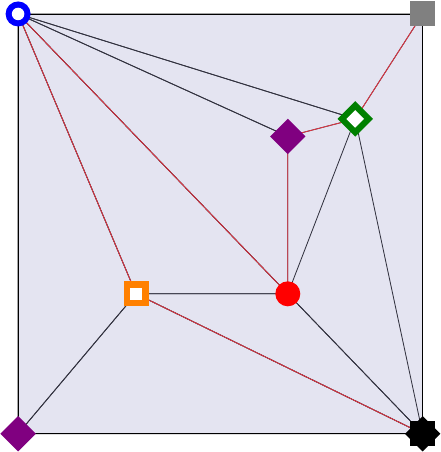}
    \caption*{Van Kampen diagram}
    \label{Figures/Fig-VK-diagram-glued}
\end{minipage}%
\begin{minipage}{.5\textwidth}
  \centering
  \includegraphics[scale=0.49]{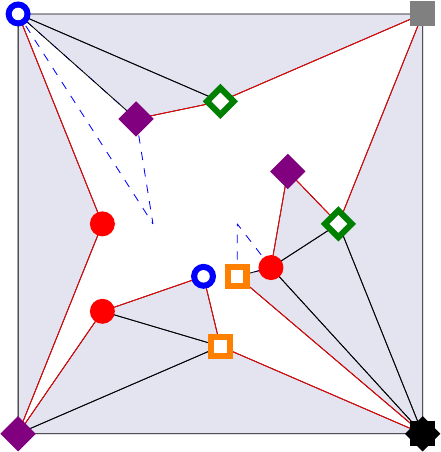}
    \caption*{Activation diagram}
    \label{Figures/Fig-activation-diagram}
\end{minipage}

\vspace{0.45cm}

\begin{minipage}{0.5\textwidth}
  \centering
  \includegraphics[scale=0.49]{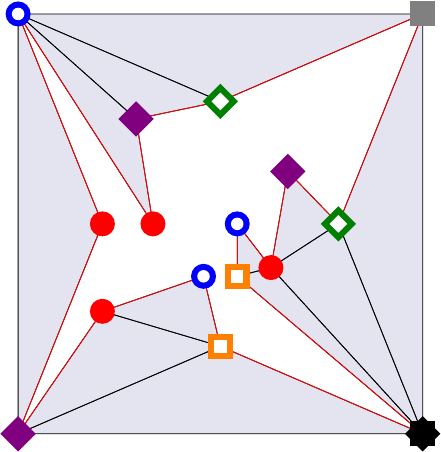}
    \caption*{Algorithmic diagram}\label{Figures/Fig-algorithmic}
\end{minipage}%
\begin{minipage}{0.5\textwidth}
  \centering
  \includegraphics[scale=0.49]{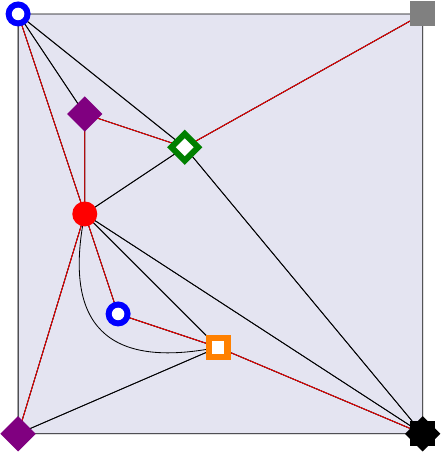}
    \caption*{Algorithmic diagram glued}\label{Figures/Fig-algorithmic-glued}
\end{minipage}
\caption{Different types of diagrams}
\label{Fig-diagrams}
\end{figure}


\section{Proof of Theorem~\ref{th:shallow}, part 2}\label{sec:main proof}

In this section, we prove the second part of Theorem~\ref{th:shallow}, which is done in the following four subsections. We start from a short sketch, disclosing the structure of this section.

In Subsection \ref{subsection:1}, we count activation processes under some restrictions. To this end, we start the first subsection by defining a nice activation process which is, in a sense, a minimal activation processes. 
Then, we explore the structure of nice processes and, using these structural results, we bound the number of such nice processes, under some restrictions.

The purpose of Subsection \ref{subsection:2} is to transfer our main counting lemma from Subsection \ref{subsection:1} --- Lemma \ref{lem: low w} --- to the probabilistic realm.
We do so by combining it with the first moment method. Informally, we show two things: First, we show that w.h.p.\ there is a `short cycle' that cannot be activated by a `fast' activation process. Second, we show that no `short cycle' can be activated by a `moderate' process (which is  `slower' than `fast' activation process, but yet quite fast). This is then used in the fourth subsection. 

Subsection \ref{subsection:3} is somewhat detached from the previous two. 
In this subsection, we develop the notion of \emph{subprocesses}. A key advantage of this notion is the ability to define, naturally, the union of two subprocesses, which is not possible for general processes. We complete this subsection by studying the behavior of the union of two nice and `contractible' subprocesses.

Finally, in Subsection \ref{subsection:4} we complete the proof of Theorem~\ref{th:shallow}. To this end, we assume that $G$ satisfies the assertions of both our main lemmas from Subsection \ref{subsection:2} --- Lemma \ref{lem: no ell-cycle between log and Klog} and Lemma \ref{lem: exists an ell-cycle with more than log} --- and show, by contradiction, that no bounded depth tree activates $G$.
The high-level idea is to consider a `short' cycle $C$ which must be activated quite slowly by a subprocess (slower than by `moderate' processes). Such a cycle exists due to our main lemmas from Subsection~\ref{subsection:2}. Then, using the activation process of $C$ from a bounded depth tree, 
 we add several branches of the tree to the cycle and get three cycles, each activated by a fast subprocess. 
  Properties of unions of subprocesses from Subsection \ref{subsection:3} imply that in such a situation, $C$ can actually be activated with at most a `moderate speed', which is a contradiction.


\subsection{Counting activation processes}\label{subsection:1}

The main aim of this section is to prove an upper bound on the number of labelled graphs $G$ on $[v+\ell]$ such that the following holds:
\begin{itemize}
\item The cycle $C\coloneqq (1,2,\ldots,\ell,1)$ is contained in $G$.
\item There exists a connected graph $H\subseteq G$ which is close to a tree (in terms of the edit distance) such that $H\rightarrow C$ via an activation process $(H;\Delta_1,\ldots,\Delta_s)$ exploiting all edges of $G$, i.e.  $G=H\cup\Delta_1\cup\ldots\cup\Delta_s$.
\end{itemize}


Let us start with a couple of definitions.

\begin{definition}[Activation Complex]
     Suppose that $\mathbf{A}=(H;\Delta_1,\ldots,\Delta_s)$ is an activation process. Let $K(\mathbf{A})$ be the simplicial complex consisting of the 2-dimensional simplices $\Delta_1,\ldots ,\Delta_s$, as well as all edges and vertices of the graph $H\cup\Delta_1\cup\ldots\cup\Delta_s$. 
     We call the abstract $2$-dimensional complex $K(\mathbf{A})$ the \emph{activation complex of $\mathbf{A}$}.
\end{definition}


Note that $K(\mathbf{A})$ is not necessarily the 2-dimensional clique complex of  $H\cup\Delta_1\cup\ldots\cup\Delta_s$ --- the latter graph may have some triangles that are not included as 2-dimensional faces in $K(\mathbf{A})$.


\begin{definition}[Contribution and Excess]
    Suppose that $H,G'\subseteq G$ are graphs, and that $H\stackrel{G}\to G'$ via an activation process $\mathbf{A}=(H;\Delta_1,\ldots,\Delta_s)$.
    Then, for every edge $e$ of $K(\mathbf{A})$, we define its \emph{contribution} as
    \[
        \con_{\mathbf{A}}(e)\coloneqq |\{\Delta_i:e\subseteq\Delta_i\text{ and }\Delta_i\text{ does \emph{not} activate }e\}|,
    \]
    where `$\Delta_i$ does not activate $e$' means that $e\in E(H\cup \Delta_1 \cup \cdots \cup \Delta_{i-1})$. 
    Furthermore, we define the \emph{excess} of $e$ as  
    \[
        \exc_{\mathbf{A}}(e)\coloneqq \begin{cases}
                            \con_{\mathbf{A}}(e)&\quad \text{if } e\in E(G')\setminus E(H),\\
                            \con_{\mathbf{A}}(e)-1&\quad \text{if } e\notin E(G')\triangle E(H),\\
                             \con_{\mathbf{A}}(e)-2&\quad \text{if } e\in E(H)\setminus E(G').
                         \end{cases}
    \]
    Lastly, we define the excess of the process $\mathbf{A}$ to be $\exc({\mathbf{A}})\coloneqq \sum \exc_{{\mathbf{A}}}(e)$. 
\end{definition}

In other words, the contribution of an edge is the number of activations that this edge produces. The `cheapest' way, in terms of the excess, to activate a cycle, when the `glued version' of the algorithmic $\mathbf{A}$-diagram (i.e., the edges of the boundary of the hole of the algorithmic $\mathbf{A}$-diagram are glued) is simply connected, is when this `glued version' does not have different equally labelled edges. 

\begin{definition}[Nice Process]
    Suppose that $H, G'\subseteq G$ are graphs and that $H\stackrel{G}{\rightarrow}G'$ via an activation process $\mathbf{A}$. 
    We say that the activation process $\mathbf{A}$ is \emph{nice} if every edge of $K(\mathbf{A})$ has a non-negative excess.
\end{definition}

\begin{remark} 
\label{rk:nice-minimal}
If $\mathbf{A}$ is {\it not} nice, $G'$ is a cycle, 
and $K(\mathbf{A})$ is simply connected, then $\mathbf{A}$ can be reduced to a nice process. Indeed, let $e$ be the first edge in the process with a negative excess. If it belongs to $G'$, then $e$ also belongs to $H$ and it does not activate any other edge i.e.\ $\mathrm{con}_{\mathbf{A}}(e)=0$. Therefore, $K(\mathbf{A})$ has an edge $e$ that does not belong to any 2-dimensional face but belongs to a cycle. Thus, $K(\mathbf{A})$ is not simply connected 
 --- a contradiction. 
Thus, $e\notin E(G')$. If $e\notin E(H)$, then $\mathrm{con}_{\mathbf{A}}(e)=0$, so the activation of $e$ can be just removed from the process. Finally, if $e\in E(H)$, then $\mathrm{con}_{\mathbf{A}}(e)\leq 1$. Similarly, we may assume that $\mathrm{con}_{\mathbf{A}}(e)=1$. But then we may replace $e$ in $H$ with the only edge that it activates, and remove $e$ from the process.\footnote{In the proof of Lemma~\ref{lemma: activations contain good activations} in Section~\ref{subsection:3}, we explain why $K(\mathbf{A})$ remains simply connected, and so we cannot get rid of all edges of $K^{(1)}(\mathbf{A})\setminus H$ eventually.} We may repeat these steps until we reach a nice activation process, see details in the proof of Lemma~\ref{lemma: activations contain good activations}. 
\end{remark}

We also observe that, if $H$ is not connected, then its connected components activate vertex-disjoint graphs --- an edge between connected components of $H$ cannot appear at any activation step. Thus, if the process is nice, then $G'$ is also disconnected --- otherwise, the last edge that was activated from a connected component of $H$ that does not overlap with $G'$ has a negative excess. 
\begin{observation}\label{obs:nice_connected}
If $\mathbf{A}$ is nice and $G'$ is connected, then $H$ is also connected.
\end{observation}

Our next lemma 
 gives a lower bound on the number of edges in a nice activation process of a cycle.
To this end, we introduce an additional notation. For a graph $H$, we denote by 
$$
d_H\coloneqq |E(H)|-|V(H)|+1
$$ 
the {\it complexity} of $H$ i.e.\ its edit distance from the closest tree. 

\begin{lemma}\label{lem: number of edges in minimal process}
    Suppose that $C$ is an $\ell$-cycle and $H{\rightarrow}C$ via a nice $\mathbf{A}=(H;\Delta_1,\ldots,\Delta_s)$. Let $G=K^{(1)}(\mathbf{A})$. Then,   
    \[
        |E(G)|\geq 3|V(H)|+3d_H+\exc({\mathbf{A}})-(3+\ell).
    \]
\end{lemma}


\begin{proof}    
    The proof follows from a double counting argument. We count the set of pairs $(\Delta_i,e)$ such that $e\subset\Delta_i$ and $\Delta_i$ does not activate $e$.
    
    On the one hand, recalling that $s$ is the number of activation steps in $\mathbf{A}$, the number of such pairs is exactly 
    \[
        2s=2(|E(G)|-|E(H)|)= 2|E(G)|-2|E(H)|.
    \]
    On the other hand, the number of such pairs equals 
    \begin{align*}
        \sum_{e\in G} \con_{\mathbf{A}}(e)&=\sum_{e\in C\setminus H} \exc_{\mathbf{A}}(e)+\sum_{e\not\in C\triangle
        H} (1+\exc_{\mathbf{A}}(e))+\sum_{e\in H\setminus C}(2+ \exc_{\mathbf{A}}(e))\\
        &=|E(G)|-|E(H\cup C)|+|E({C\cap H})|+2|E({H\setminus C})|+\exc({\mathbf{A}})\\
        &=|E({G})|-|E({H\setminus C})|-|E({C})|+|E({C\cap H})|+2|E({H\setminus C})|+\exc({\mathbf{A}})\\
        &=|E(G)|+|E({H})|-\ell+\exc({\mathbf{A}}).
    \end{align*}
    Due to Observation~\ref{obs:nice_connected}, the graph $H$ is connected. Thus, $|E(H)|\geq |V(H)|-1$ and hence we conclude that 
    \[
        |E(G)|=3|E(H)|+\exc({\mathbf{A}})-\ell\geq 3|V(H)|+3d_{H}+\exc({\mathbf{A}})-(3+\ell) .\qedhere
    \]
\end{proof}


Let $T$ be a tree and $W$ be a walk on $T$ that visits each edge of $T$ at most twice. Then, a straightforward induction argument in the number of steps of the walk implies that the set of edges of $T$ that are visited exactly once is a path. In the next lemma, we generalise this statement for several walks on a tree. 
We will use this generalisation later in this section in the proof of our main counting result, Lemma \ref{lem: low w}.

\begin{lemma}\label{lem: union of walks covering a tree} 
    Suppose that $W_1,W_2,\ldots ,W_k$ are walks on trees $T_1\subseteq T_2\subseteq \ldots \subseteq T_k$, respectively.
    Further, suppose that, for every $1\leq i\leq k$, each edge of $T_i$ is covered by the union of walks $W_1\cup \ldots \cup W_i$ either once or twice.
    Moreover, for every $1\leq i\leq k$, set $F_i$ to be the set of edges of $T_i$, covered exactly once by the union of walks $W_1\cup\ldots \cup W_i$.  
    Then, for every $1\leq i\leq k-1$, the set of edges $E(F_{i})\cap E(W_{i+1})\subseteq T_{i}$ forms a path. 
\end{lemma}

\begin{figure}[ht]
\centering
\includegraphics[scale=0.49]{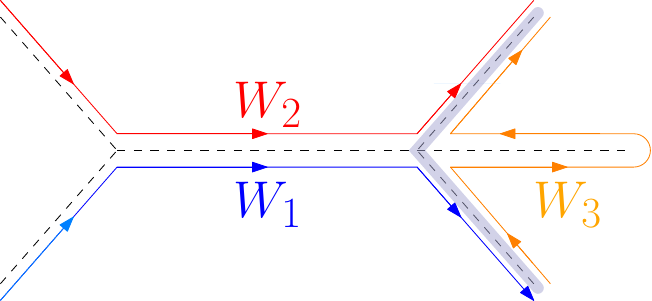}
  \caption{Three walks $W_1,W_2,$ and $W_3$, covering each edge of $T_3$ once or twice, and $E(F_2)\cap E(W_3)$ are highlighted.}
\end{figure}

\begin{proof}
    Fix $1\leq i\leq k-1$. First, we prove that the graph $H$ consisting of the edges of $E(F_{i})\cap E(W_{i+1})$ has maximum degree $2$.
    Second, we show that $H$ is connected, concluding the proof of the lemma, as $H$ is a subgraph of the tree $T_i$, and hence does not contain any cycle.

    Indeed, suppose towards contradiction that there is a vertex $v\in V(H)$ with $\deg(v)\geq 3$. Let $w_1,w_2,w_3$ be three distinct neighbours of $v$ in $H$.
    Hence, $\{v,w_p\}\in E(F_{i})\cap E(W_{i+1})$ for $p=1,2,3$.
    
    Recall that for any fixed $1\leq j\leq k$, each edge of $T_{j}$ is covered at most twice by the union of walks $W_1\cup \ldots \cup W_{j}$. Hence, by the definition of $F_i$, for $p=1,2,3$, the edges $\{v,w_p\} \in E(F_{i})\cap E(W_{i+1})$ are covered exactly once by the union of walks $W_1\cup \ldots \cup W_{i}$, and exactly once by the walk $W_{i+1}$.
    Without loss of generality, suppose that for every $p<q$, the edge $\{v,w_p\}$ appears before the edge $\{v,w_q\}$ along the walk $W_{i+1}$. Every part of the walk $W_{i+1}$ from $w_1$ to $w_2$ uses both edges $\{v,w_1\}$ and $\{v,w_2\}$.
    This is because the unique path between $w_1$ and $w_2$ in $T_{i+1}$ is $w_1vw_2$. 
    In the same manner, every part of the walk $W_{i+1}$ from $w_2$ to $w_3$ uses both edges $\{v,w_2\}$ and $\{v,w_3\}$. We conclude that $W_{i+1}$ covers the edge $\{v,w_2\}$ twice --- a contradiction. 
    
    Next, suppose towards contradiction that $H$ is not connected. Let $u,v\in V(H)$ be two vertices that are not connected in $H$ with the minimum distance in $T_{i}$. 
    Since $T_{i}\subseteq T_{i+1}$ are trees, there is a unique path $P\subseteq T_{i+1}$ between $u,v$, which is the same in both trees. Let $u'$ be the neighbour of $u$ in $P$. 
     In particular, $\{u,u'\}\in E(T_{i})$ and hence it is covered at least once by the union of walks  $W_1 \cup \ldots \cup W_{i}$. 
    Note also that the edge $\{u,u'\}$ appears at least once in $W_{i+1}$. This is due to the connectivity of a graph explored by a walk: we know that $W_{i+1}$ is a walk along the tree $T_{i+1}$ that meets both $u$  and $v$, and so it passes through $P$. 
    This implies that actually, the edge $\{u,u'\}$ is covered exactly once by $W_{i+1}$ and exactly once by the union of walks $W_1 \cup \ldots \cup W_{i+1}$ (since the total number of instances of every edge in all walks is at most $2$). 
    Thus, $\{u,u'\}\in E(H)$. In particular, the vertices $u,u'$ lie in the same connected component of $H$ and $v$ lies in a distinct connected component of $H$.
    This contradicts the minimality assumption on $u,v$, as $\dist_{T_{i}}(u',v)<\dist_{T_{i}}(u,v)$. 
\end{proof}

Next, we introduce a compressed version of the algorithmic van Kampen diagram, we call such diagrams `activation diagrams'.
The idea behind the following definition is to avoid repetitions of faces: the number of 2-dimensional faces in the compressed version equals the number of triangles that participate in the activation process, while the number of faces in an algorithmic diagram may be exponentially larger than the number of faces in the respective process.

\begin{definition}[The Activation Diagram]
    Suppose that $H\stackrel{G}\to C$ via $\mathbf{A}=(H;\Delta_{1},\ldots ,\Delta_{s})$ where $H\subseteq G$ are connected graphs, and $C$ is a cycle.
    Moreover, for all $1\leq k\leq s$ denote by $e_k=\{x_k^1,x_k^2\}$ the edge that was activated at step $k$, i.e.\ $e_k\coloneqq \Delta_{k}\setminus \left(H\cup \bigcup_{i=1}^{k-1}\Delta_i\right)$. Define the following sequence of labelled complexes iteratively:
    \begin{enumerate}
        \item Start by initialasing $D_{s+1}\coloneqq C$, and mark all edges as \emph{marked} edges.
        \item  Let $k\geq 1$ and suppose that $D_{k+1}$ is defined. Then, define the complex $D_{k}$ by gluing the triangle $\Delta_{k}=e_k\cup\{y_k\}$, to the unique \emph{marked} edge of $D_{k+1}$ whose vertices are labelled by $x_k^1$ and $x_k^2$. Then, label the third vertex of $\Delta_k$ by $y_k$, unmark the edge $\{x_k^1,x_k^2\}$, and for $i=1,2$ mark the edge $\{x_k^i,y_k\}$ if in $D_{k+1}$ there is no marked edge whose vertices have labels $x_k^i,y_k$.
    \end{enumerate}
    We call the resulting complex $D_1$, the \emph{$\mathbf{A}$-activation diagram} of $C$; when clear from the context we omit the dependency in $\mathbf{A}$ (see Figure~\ref{Fig-diagrams}).
\end{definition}


Hence, the complexes $D_i$ are homotopy equivalent to $S^1\times[0,1]$, that is an annulus. In other words, a planar embedding of $D_i$ is a region between two cycles --- $C$, which we call {\it the external boundary}, and the cycle, formed by edges of $H$ as well as some edges of $G\setminus H$ that were used in more than one activation, which we call {\it the internal boundary}. 

The next lemma connects the excess function of a nice process and activation diagrams.

\begin{lemma}
    Assume that $H\stackrel{G}{\rightarrow} C$ via a nice process $\mathbf{A}=(H;\Delta_1,\ldots ,\Delta_s)$ where $C$ is a cycle.
    Then the internal boundary of the activation diagram has $2|E(H)|+\exc({\mathbf{A}})=2(|V(H)|-1)+2d_{H}+\exc({\mathbf{A}})$ vertices\footnote{Vertices with equal labels are considered as different vertices here.}.
\label{lemma:internal-length}
\end{lemma}

\begin{proof}
First, we note that 
\begin{align*}
     2|E(H)|+\exc({\mathbf{A}})
     &= \sum_{e\in E(H)}(\exc_{\mathbf{A}}(e)+2)+\sum_{e\in E(K(\mathbf{A}))\setminus E(H)}\exc_{\mathbf{A}}(e)\\
     &=\sum_{e\in E(H)\setminus E(C)}\con_{\mathbf{A}}(e)+\sum_{e\in E(H)\cap E(C)}(\con_{\mathbf{A}}(e)+1)+\sum_{e\in E(K(\mathbf{A}))\setminus E(H)}\exc_{\mathbf{A}}(e).
\end{align*}

Hence, it is enough to prove the following:
\begin{enumerate}[label=\textbf{(C\arabic*)}]
    \item \label{item: case 1} Every $e\in E(H)$ appears exactly $\con_{\mathbf{A}}(e)+\mathbbm{1}_{e\in E(C)}$ times along the internal boundary.
    \item \label{item: case 2} Every $e\in E(K(\mathbf{A}))\setminus E(H)$ appears exactly $\exc_{\mathbf{A}}(e)$ times along the internal boundary.
\end{enumerate}

Recall that $\Delta_i$ {\it does not activate} $e\in E(K(\mathbf{A}))\cap \Delta_i$ if $e\in E(H)\cup \Delta_1 \cup \cdots \cup \Delta_{i-1}$. Otherwise, we say that $\Delta_i$ {\it activates} $e$. With this terminology, for every $e\in E(K(\mathbf{A}))\setminus E(C)$, the number of edges labelled as $e$ in the activation diagram $D_1$ is the number of triangles $\Delta_i$ containing $e$ but not activating $e$. This is exactly $\con_\mathbf{A}(e)$.
Furthermore, as each edge $e\in E(C)$ is presented in $D_{s+1}$, the number of edges labelled as $e$ in $D_1$ is $\con_{\mathbf{A}}(e)+1$. We further prove separately \ref{item: case 1} and \ref{item: case 2}.

\paragraph{Proof of \ref{item: case 1}:}
Let $e\in E(H)$. As mentioned, the number of edges labelled as $e$ in $D_1$ is $\con_\mathbf{A}(e)+\mathbbm{1}_{e\in E(C)}$.
Since $e\in E(H)$, its contribution in $\mathbf{A}$ equals to the number of triangles $\Delta_i$ containing $e$.
Hence, to conclude the proof, it is enough to show that every instance of $e$ in $D_1$ belongs to the internal boundary.

To this end, note that, in the process generating $D_1$, once an edge is added, it becomes an edge of the internal boundary.
Thus, for an edge not to be on the internal boundary, there must be some triangle $\Delta_i$ activating it. Since $e\in E(H)$, no triangle $\Delta_i$ activates $e$, and thus every copy of $e$ lies at the internal boundary.

\paragraph{Proof of \ref{item: case 2}:}
Let $e\in E(K(\mathbf{A}))\setminus E(H)$. Recalling that $e$ appears $\con_{\mathbf{A}}(e)+\mathbbm{1}_{e\in E(C)}$ times in $D_1$ and that $\con_{\mathbf{A}}(e)=\exc_{\mathbf{A}}(e)+\mathbbm{1}_{e\not\in E(C)}$, we find that $e$ appears $\exc_{\mathbf{A}}(e)+\mathbbm{1}_{e\not\in E(C)}+\mathbbm{1}_{e\in E(C)}=\exc_{\mathbf{A}}(e)+1$ times in $D_1$. Hence, it suffices to show that, outside of the internal boundary, there is a single edge labelled as $e$.

Indeed, as before, in the process generating $D_1$, once an edge is added, it becomes an edge of the internal boundary. As $e\in E(K(\mathbf{A}))\setminus E(H)$ there is a triangle $\Delta_i$ that activates $e$ --- clearly, this triangle is unique. For any $j>i$, all edges labelled $e$ in $D_j$ also belong to the internal boundary of $D_j$. As $e$ is activated through $\Delta_i$, a single edge labelled $e$ belonging to the internal boundary of $D_{i+1}$, becomes a non-internal boundary edge in $D_{i}$. 
 Further, for all $j<i$ we have $e\not \in \Delta_j$. Thus, the number of edges labelled $e$ in the internal boundary of $D_j$ equals the number of edges labelled $e$ in $D_i$.
In particular, the above holds for $j=1$, completing the proof. 
\end{proof}

To state the main lemma in this subsection, we introduce one more notation. Fix non-negative integers $w$ and $v\geq\ell\geq 3$. We let $P_{v,w}:=P_{v,w}(\ell,n)$ be the collection of all $\mathbf{A}$-activation diagrams $D$ such that the following holds:
\begin{itemize}
    \item $\mathbf{A}=(H;\Delta_1,\ldots,\Delta_s)$ is a nice activation process of an $\ell$-cycle;
    \item $V(H)\subseteq[n]$ and $|V(H)|=v$;
    \item $2d_{H}+\exc({\mathbf{A}})=w$.
\end{itemize}

\begin{lemma}\label{lem: low w}
    Let $\ell=\ell(n)\geq 3$, $v=v(n)\geq\ell(n)$, and and $w=w(n)$ be three sequences of non-negative integers. 
    Then, 
    \[|P_{v,w}|\leq  n^v (2v)^{4w} w^w \cdot \binom{2v+w-2}{w}2^{7v}.\]
\end{lemma}
\begin{proof}
    Let $B_{0},\ldots, B_{\ell-1}$ be triangulated discs without internal vertices such that $\sum_{i=1}^{\ell}|V(B_i)| = 2(v-1)+w+\ell$. 
    In each $B_i$, let $\{x_i,y_i\}$ be a boundary edge, and let $D= D(B_0,\ldots,B_{\ell-1})$ be the complex defined by gluing $x_i$ and $y_{i+1}$, where addition is done modulo $\ell$. 
    Note that $D$ is a 2-dimensional simplicial complex which is homotopy equivalent to $S^1\times[0,1]$ with $|V(D)|=2(v-1)+w$.
    Let $V\subseteq [n]$ be a set of size $v$, and let $T$ be a tree on $V$.
    We say that $(D,T,V)$ equipped with a labelling $V(D)\to V$ is an \emph{activation scheme} if the following conditions hold (see Figure~\ref{Fig-activation-scheme}):
    %
    \begin{enumerate}[label=\textbf{(D\arabic*)}]
        \item\label{item 1 - X} The internal boundary of $D$ has length $2(v-1)+w$ and the external boundary of $D$ has length~$\ell$.  
        \item \label{item 2 - X} A subset $E^*$ of the set of edges of the internal boundary of $D$ of size $2(v-1)$ is specified; every edge from $E^*$ is labelled as some edge of $T$. 
        \item\label{property of the diagram} Each edge $e\in E(T)$ has exactly two edges from $E^*$ labelled as $e$.
    \end{enumerate}

\begin{figure}[ht]
\centering
\begin{minipage}{.5\textwidth}
  \centering
    \includegraphics[scale=0.49]{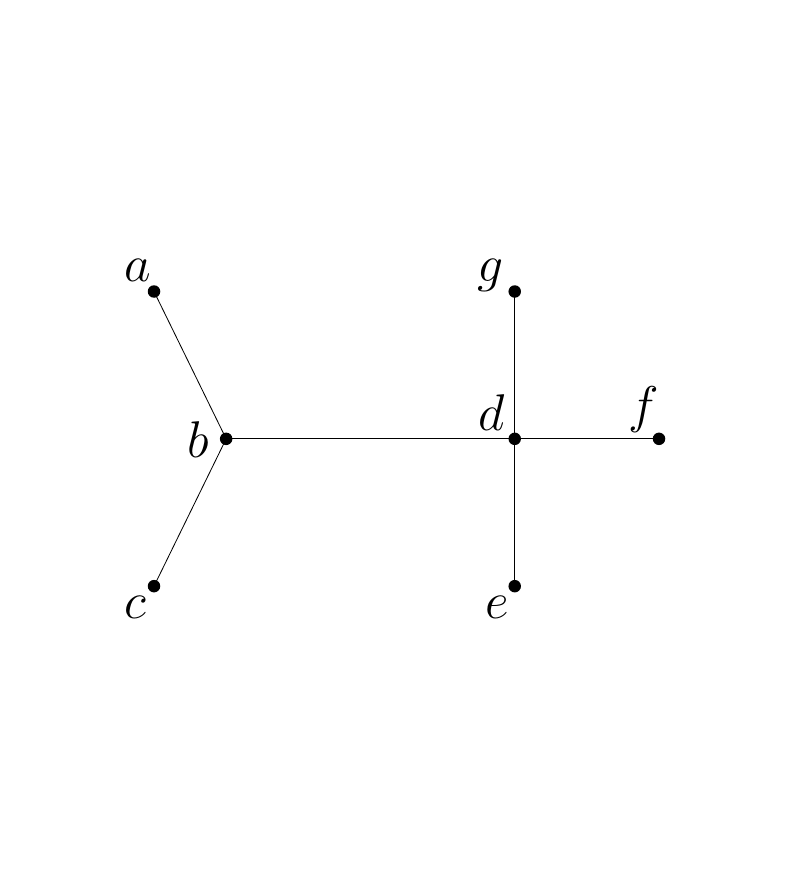}
  \caption*{The tree $T$ with $V(T)\coloneqq \{a,b,c,d,e,f,g\}$}
\end{minipage}%
\begin{minipage}{0.5\textwidth}
  \centering
  \includegraphics[scale=0.49]{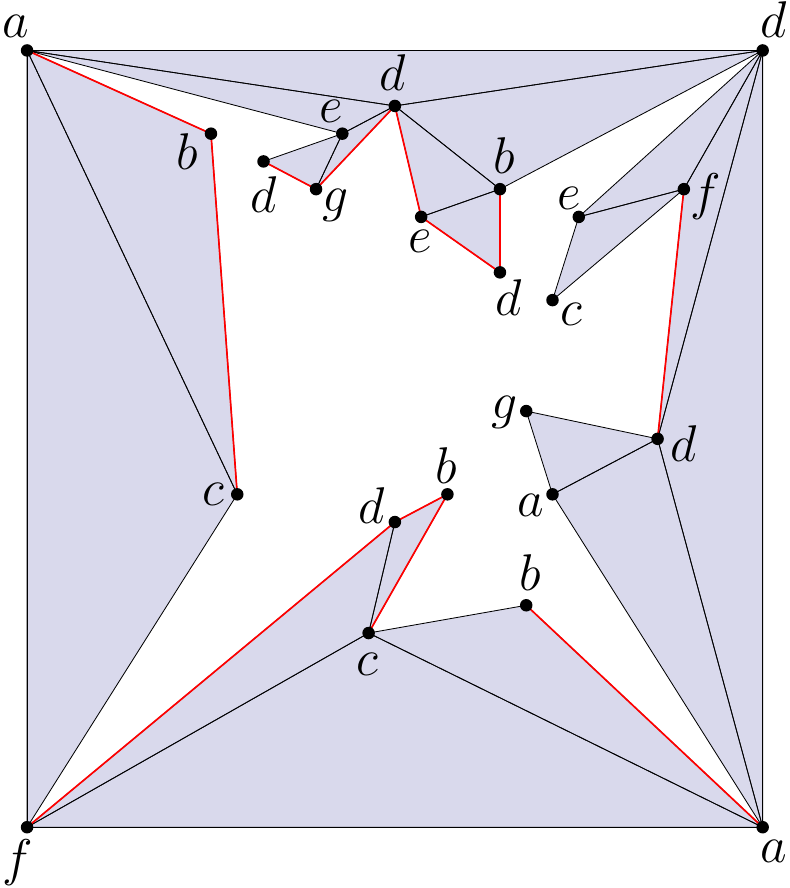}
  \caption*{The complex $D$ labelled by vertices from $V$, where the set $E^*$ is coloured red}
\end{minipage}%
\caption{An example of an activation scheme} \label{Fig-activation-scheme}
\end{figure}

    
    Let $X$ be the set of all activation schemes.
    We claim that $|P_{v,w}|\leq |X|$, and hence it suffices to bound $|X|$.
    Indeed, it is easy to map $P_{v,w}$ into $X$ injectively. First of all, for every labelled connected graph $H$, let $T_H$ be the canonical spanning subtree of $H$. For $D\in P_{v,w}$, let $\mathbf{A}_D$ be the respective activation process that starts from the graph $H_D$ comprising some of the edges of the internal boundary of $D$. Let us show that the map $D\to(D,T_{H_D},V(D))$ is the desired injection into $X$.
    If $D\neq D'$, then also $(D,T_{H_D},V(D))\neq (D',T_{H_D'},V(D'))$.
    Let us show that the image of $P_{v,w}$ lies in $X$. Indeed, $D\in P_{v,w}$ is isomorphic to some $D(B_0,\ldots , B_{\ell-1})$, and its set of vertices $V(D)=V(H_D)$ has size $2(v-1)+w$ and satisfies \ref{item 1 - X} by the definition of an activation diagram and Lemma~\ref{lemma:internal-length}. 
    Further, since $\mathbf{A}_D$ is a nice activation process and $T_{H_D}$ is a spanning subtree of $H_D$, which satisfies $V(H_D)=V(D)$, every edge of $T_{H_D}$ appears at least twice along the internal boundary of $D$. This implies that both $\ref{item 2 - X}$ and \ref{property of the diagram} are satisfied for any subset of the internal boundary $E^*$ comprising exactly two instances of every edge of $T_{H_D}$.
    
    Hence, to conclude the proof of the lemma, it is indeed enough to bound $|X|$. 
    To bound the number of activation schemes $(D,T,V)$, we first note that there are \textbf{at most $\binom{n}{v}$ ways to choose $V$}.
    Given $V$, we now bound the number of pairs $(D,T)$.

    To this end, we first bound the number of choices of an isomorphism type for $D$, and then bound the number of labellings of $V(D)$ with labels taken from $V$ 
     and choices of a tree $T$ on $V$ so that the triple $(D,T,V)$ is a valid activation scheme.

  
    Indeed, every isomorphism type of $D=D(B_0,\ldots,,B_{\ell-1})$ can be achieved in the following way:
    First, choose a sequence of integers $(a_0,\ldots ,a_{\ell-1})$ such that $a_i\geq 2$ and moreover, $\sum_{i=0}^{\ell-1}a_i=2v+w+\ell-2$. Then, fix an $\ell$-cycle $\partial D$ on the role of the outer boundary of $D$, and, for every $i$, on the top of its $i$-th edge, glue an $a_i$-cycle, denoted by $\partial B_i$.
    Lastly, choose isomorphism types of triangulations with boundary $\partial B_i$ and no internal vertices.
    To bound the number of such options, recall that the number of isomorphism types of triangulations of an $(n+2)$-cycle without internal vertices is exactly the $n$-th Catalan number denoted by $C_n$.
    Thus, the number of ways to choose an isomorphism type for $D$ is $\sum \prod_{i=0}^{\ell-1} C_{a_i-2}$, where the sum is taken over all $a_0,a_1,\ldots,a_{\ell-1}\geq 2$ which sums up to $2v+w+\ell-2$.
    Equivalently, the number of ways to choose an isomorphism type for $D$ is $\sum \prod_{i=0}^{\ell-1} C_{a_i-1}$, where the sum is taken over all $a_0,a_1,\ldots,a_{\ell-1}\geq 1$ which sums up to $2v+w-2$.
    By the classical result of Catalan \cite{Cat}, we have
    \[
        \sum_{\stackrel{a_0+\ldots +a_{\ell-1} =2v+w-2}{a_0,\ldots ,a_{\ell-1}\geq 1}} C_{a_0-1}\cdots C_{a_{\ell-1}-1} = \frac{\ell}{2(2v+w-2)-\ell}\binom{2(2v+w-2)-\ell}{\ell}.
    \]
    Hence, we conclude that there are \textbf{at most $\frac{\ell}{4v+2w-\ell-4}\binom{4v+2w-\ell-4}{\ell}\leq 2^{4v+2w}$ ways to choose an isomorphism type} for $D$. Fix an isomorphism type of $D$.
  
    
    
    Since the internal boundary is a spanning subgraph of $D$, it is enough to bound the number of choices of $T$ and assignments of labels for the vertices of the internal boundary such that the resulting triple $(D,T,V)$ is an activation scheme.
    Clearly, there are \textbf{at most $\binom{2v+w-2}{2v-2}$ ways to choose the set of edges $E^*$} from the internal boundary; fix such a choice. There are \textbf{at most $v^{2w}$ ways to choose labels} from $V$ for the endpoints of the edges of the internal boundary that do not belong to $E^*$.
    
    Let us summarise that the number of ways to choose $V$, the isomorphism type of $D$, the set $E^*$, and labels for the vertices of the edges of the internal boundary of $D$ that do not belong to $E^*$ is at most 
    \begin{align}\label{eq: intermidiate counting}
       \binom{n}{v}\binom{2v+w-2}{w}\cdot v^{2w} \cdot 2^{4v+2w}\leq \frac{(2v)^{2w}\cdot n^{v}}{ v!}\binom{2v+w-2}{w} \cdot 2^{4v}. 
    \end{align}

    Given the above information, it is left to bound the number of choices of $T$ and valid assignments of labels for the endpoints of edges from $E^*$ (note that some of them might be already labelled). 
    In other words, we need to choose $T$ and a mapping from $E^*$ to $E(T)$ such that every edge of $T$ has two pre-images.
    We claim that this can be done in at most $2^{3v}\cdot v^{2w}\cdot w!\cdot v!$ ways. If this is true, then combining it with \eqref{eq: intermidiate counting}, we obtain the required inequality
    \[
        |P_{v,w}|\leq |X|\leq n^v (2v)^{4w} w^w \cdot \binom{2v+w-2}{w}\cdot 2^{7v}.
    \]
    
    We start by fixing an orientation of the internal boundary, which can be done in $2$ ways. The set $E^*$ partitions the internal boundary into ordered intervals (the order is induced by the above orientation): directed paths comprising edges of $E^*$ (we refer to them as $E^*$-paths) followed by directed paths comprising edges that do not belong to $E^*$. Let $J_{1},\ldots,J_{k}$ be all $E^*$-paths. Note that $k\leq w$, as the number of edges that do not belong to $E^*$ is $w$. For a valid choice of $T$ and the labels of the remaining vertices of $D$, the labelled path $J_i$ corresponds to a walk on $T$ with the property that no edge of $T$ is covered more than twice; we denote this walk by $W_i$. 
    Let us denote by $T_i'$ the subtree of $T$ that is covered by the edges of $W_i$. 
    By \ref{property of the diagram} the walks $W_1,\ldots ,W_k$ cover all edges of $T$, each edge exactly twice.
    Since $T$ is connected, one can find a permutation on $[k]$ denoted by $\sigma$, so that for all $t\in[k]$, the graphs $T_t\coloneqq \cup_{i=1}^{t}T_{\sigma(i)}'$ are nested subtrees of $T$.
    Note also that the starting point and the ending point of every $W_i$ are already labelled, as they are also endpoints of edges that do not belong to $E^*$. 
    
    Thus, in order to bound the number of possible choices of $T$ and labellings of the remaining set of vertices of $D$, it is enough to 1) choose a permutation $\sigma$ on $[k]$, which can be done in {\bf at most $k!\leq w!$ ways}; 2) choose a sequence of walks $(W_i)_{i=1}^{k}$ such that for all $i\in[k]$, the graph $T_i'\coloneqq E(W_i)$ is a tree on some subset of $V$ and
    \begin{enumerate}[label=\textbf{(P\arabic*)}]
        \item \label{property 1} For every $i\in[k]$, each edge of $T_i\coloneqq \cup_{j=1}^{i}T_j'$ is a tree.
        \item \label{property 2} Each edge of $T\coloneqq \cup_{i=1}^{k}T_i'$ is covered exactly twice by the union of edges induced by the walks $W_1,\ldots , W_k$.
        \item \label{property 3} For every $i\in[k]$, the walk $W_i$ is of length $|E(J_i)|$.
    \end{enumerate}
    The rest of the proof is devoted to proving that there are at most $2^{3v}\cdot v^{2w}\cdot v!$ ways to construct such sequences, concluding the proof of Lemma \ref{lem: low w}. 

    For every $i\in[k]$ we shall identify a sequence of vertices $W_i= (w_0^{i},\ldots ,w_{|V(J_i)|}^{i})$ forming a walk, where $w_0^{i}$ and $w_{|V(J_i)|}^{i}$ are already specified. 
    In every valid sequence of walks $(W_i)_{i=1}^{k}$, i.e.\ a sequence satisfying \ref{property 1}, \ref{property 2}, and \ref{property 3}, every edge of $T$ appears twice.
    In particular, there are in total $2v-2$ edges along all walks $(W_i)_{i=1}^{k}$. We denote the $s$-th edge $\{w^i_s,w^i_{s+1}\}$ of $W_i$ by $e_s^i$. An edge $e_s^i$ is {\it fresh}, if it does not appear in the set of edges $e^j_{\ell}$ preceding $e_s^i$ in the lexicographic order on $(j,\ell)$, $j\in[k]$, $\ell\in [|V(J_j)|]$. Clearly, exactly $v-1$ edges are fresh. We will also refer to non-fresh edges as \emph{old} edges.

    \begin{itemize}
    
    \item We choose a set $\Tilde{E}\subseteq E^*$ of $v-1$ edges that play the role of fresh edges. There are {\bf at most $2^{2v}$ choices} of this set. We remark that this also determines the set of old edges.

    \end{itemize}

    In every valid sequence of walks $(W_i)_{i=1}^{k}$, for every $i\in[k]$, let the forest $F_{i}\subseteq T_{i}$ consist of edges covered once by $W_1,\ldots, W_{i}$. Clearly, $P_1:=F_1$ is a path. Moreover, by Lemma \ref{lem: union of walks covering a tree}, the graphs $P_{i}\coloneqq E(F_{i-1})\cap E(W_i)\subseteq T_{i-1}$ are paths for all $i\in\{2,\ldots,k\}$. 
    
    \begin{itemize}

    \item For each $i\in[k]$, there are at most $2^{|E(J_{i})\cap \Tilde{E}|}$ ways to specify which edges of the (still undefined) walk $W_i$ are participating in $P_{i}$. We call these edges \emph{path edges}. We choose the path edges from $E^*$ in {\bf at most $2^{\sum |E(J_i)\cap \Tilde{E}|}=2^{v-1}$ ways}.
    

    \item Clearly, in a valid sequence of walks, there are {\bf at most $v!$ possible sequences} of labels for the latter (according to the orientation that we fixed) endpoint $w_{s+1}^{i}$ of every fresh edge $e_{s}^i$. Fix such a sequence of labels.

    \end{itemize}

    We claim that given the above information, there are {\bf at most $v^{2w}$ labellings of $D$} that are consistent with this information.
    This will conclude the proof as then the total number of sequences $(W_i)_{i=1}^{k}$ satisfying \ref{property 1}, \ref{property 2}, and \ref{property 3} is at most $2^{3v}\cdot v^{2w}\cdot v!$ as required.

    To this end, let us show by induction that, for every $i\in[k]$, given $W_1,\ldots,W_{i-1}$, there are at most $v^2$ ways to choose $W_i$ such that there exists an extension of the sequence $W_1,\ldots,W_i$  satisfying \ref{property 1}, \ref{property 2}, and \ref{property 3}.
    
    Indeed, suppose that either $i=1$ or $i\geq 2$ and $W_1,\ldots,W_{i-1}$ are chosen. For clarity of presentation, from now on, when it is clear from the context, we will omit the upper indices of vertices and edges in walks.
    
    Assuming $i\geq 2$, we recall that $W_1,\ldots,W_{i-1}$ define the trees $T'_j$ for every $j\in[i-1]$ and, hence, the tree $T_{i-1}=\cup_{j=1}^{i-1} T_j'$. Given two vertices of the forest $F_{i-1}\subseteq T_{i-1}$, which belong to the same connected component of $F_{i-1}$, there is a unique path between them. 
    In particular, there are at most $v^2$ ways to choose $P_{i}$ --- it is sufficient to choose two vertices in $F_{i-1}$ to define it. Fix a path $P_i\subseteq F_{i-1}$. Since the path edges of $W_i$ were already specified, and we have just defined labels of all vertices along $P_i$, we can label the respective vertices of path edges in $W_i$. 

    It is left to show that given the above, there is at most a single way to label vertices of old non-path edges of $W_i$. We label these vertices one by one as they appear in $W_i$.
    Let $w_{s}$ be the first non-labeled vertex in $W_i$. 
    As the first vertex in $W_i$ had received a label, we conclude that $s\geq 2$. 
    The edge $e_{s-1}=\{w_{s-1},w_{s}\}$ is an old non-path edge since otherwise $w_s$ has been already labelled.
    Therefore, its first fresh instance appears in the same walk $W_{i}$ --- denote it by $e_{p}$, where $p\leq s-2$.
    We claim that either there is no valid extension of $(W_1,\ldots,W_{i-1})$ or that $p$ is maximal among all $q\leq s-2$ such that the edge $e_{q}$ is a fresh edge, and its second appearance is not any of $e_{r}$ for $q\leq r\leq s-2$.
    Indeed, otherwise, there is $q\in\{p+1,\ldots,s-2\}$ such that $e_{s-1}$ is the second appearance of $e_{p}$ and $e_{q}$ is a fresh edge such that its second instance does not appear before $e_{s-2}$ in $W_i$.
    Since $W_{i}$ is a walk on a tree, the sequence $W'=(w_{p},w_{p+1},\ldots, w_{q},w_{q+1} ,\ldots ,w_{s-1},w_{s})\subseteq W_{i}$ is also a walk on a tree.
    By our assumption, $\{w_{p},w_{p+1}\}=\{w_{s-1},w_{s}\}$, and thus, we actually have $w_{p+1}=w_{s-1}$. Hence, $W''=(w_{p+1},\ldots, w_{q},w_{q+1} ,\ldots ,w_{s-1})\subseteq W'$ is a closed walk on a tree.
    In particular, each edge in the walk $W''\subseteq W_i$ appears twice --- this is a contradiction to the assumptions on $e_q$.

    Now, either there is no valid extension, and we are done; or $p$ is maximal among all $q\leq s-2$ such that the edge $e_{q}$ is a fresh edge, and its second instance is not any of $e_{r}$ for $q\leq r\leq s-2$.
    In the latter case, the edge $e_p=\{w_p,w_{p+1}\}$ is uniquely determined and then we set $w_s:=w_p$ as the only possible label for the latter endpoint of $e_p$. 
    
    Hence, there are at most $v^2$ ways to extend $(W_1,\ldots, W_{i-1})$ to $(W_1, \ldots , W_{i})$, and so there are at most $v^{2k}\leq v^{2w}$ ways to define $W_1, \ldots , W_k$ as required.
\end{proof}


\subsection{Activation processes are typically slow} \label{subsection:2}

In this section, we transfer our main counting Lemma \ref{lem: low w} into the probabilistic realm. Let $\mathbf{G}\sim G_{n,p}$, 
 where $p<cn^{-1/3}$ and $c\in(0,2^{-7/3})$ is a constant. 

Let $\ell \coloneqq \ell(n)\geq 3$ be a sequence of integers, and also let $v=v(n),w=w(n)$ be two sequences of non-negative integers.
 Let $Z^{\ell}_{v,w}$ be the (random) number of activation diagrams $D\in P_{v,w}$ with $D\subseteq X^{(2)}(\mathbf{G})$. Hereinafter, with a slight abuse of notation, we write $D\subseteq X^{(2)}(\mathbf{G})$ to denote the event that all triangles in $D$ (labelled by the vertices of $\mathbf{G}$) also belong to $X^{(2)}(\mathbf{G})$. 

\begin{lemma}\label{lem: expectation of the number of activations}
Suppose that $3\leq \ell=\ell(n)\leq v=v(n)$, and $w=w(n)$ are sequences of non-negative integers such that $w\leq 2v^2$. 
 Then, for large enough $n$, 
        \[
        \E\left[Z_{v,w}^{\ell}\right] \leq  n^{(-1/3+o(1))w+(1/3 +o(1))\ell+1} \cdot v^{6w}\cdot(2^7 c^3)^{v}.
        \]
\end{lemma}

\begin{proof}  
For every $\mathbf{A}$-activation diagram $D\in P_{v,w}$, where $\mathbf{A}=(H;\Delta_1,\ldots,\Delta_s)$ is a nice activation process of an $\ell$-cycle, Lemma \ref{lem: number of edges in minimal process} and the equality $2d_H+\exc(\mathbf{A})=w$, from the definition of a diagram, imply that 
    \[
    |E(D)|\geq 3v+3d_H +\exc(\mathbf{A})-3-\ell \geq3v+w-3-\ell.
    \]
Hence, by Lemma \ref{lem: low w},

\begin{align*}
    \E\left[Z_{v,w}^{\ell}\right]&=\sum_{D\in P_{v,w}}\mathbb{P}(D\subseteq X^{(2)}(\mathbf{G}))= \sum_{D\in P_{v,w}}p^{|E(D)|} \leq  |P_{v,w}|\cdot p^{3v+w-3-\ell}\\
    &\leq  n^v (2v)^{4w} w^w\cdot  \binom{2v+w-2}{w}\cdot 2^{7v} \cdot \left(cn^{-\frac{1}{3}}\right)^{3v+w-3-\ell}\\
    &\leq (2v)^{4w} w^w \cdot  \binom{2v+w-2}{w} \cdot \left(cn^{-\frac{1}{3}}\right)^{w-3-\ell} \cdot (2^7c^3)^{v}.
 \end{align*}
It remains to apply the inequality
\begin{equation*}
 \binom{2v+w-2}{w} \leq \begin{cases} (3v)^{w} \quad&\text{ if }w\leq v,\\  2^{3w} \quad & \text{ if }w>v.
\end{cases}
\end{equation*}

Indeed, if $w\leq v$ then
\begin{align*}
       \E\left[Z_{v,w}^{\ell}\right]&\leq (2v)^{4w}w^w \cdot (3v)^{w} \cdot (cn^{-1/3})^{w-3-\ell}\cdot  (2^7 c^3)^{v}\\
       &\leq  n^{(-1/3+o(1))w+(1/3 +o(1))\ell+1} \cdot v^{6w} \cdot(2^7 c^3)^{v}.
\end{align*}
Moreover, if $v<w\leq 2v^2$ then
\begin{align*}
       \E\left[Z_{v,w}^{\ell}\right]&\leq (2v)^{4w}w^w \cdot 2^{3w} \cdot (cn^{-1/3})^{w-3-\ell} \cdot(2^7 c^3)^{v}\\
       &\leq n^{(-1/3+o(1))w+(1/3+o(1))\ell+1} \cdot v^{6w} \cdot (2^7 c^3)^{v}.\qedhere
\end{align*}
\end{proof}

As mentioned at the beginning of Section \ref{sec:main proof}, in the next two lemmas we show the following:
First, we show that typically, for $\ell\leq n^{1/18-\Theta(1)}$, every $\ell$-cycle in $\mathbf{G}$, cannot be activated by an activation process using only a `moderate amount of vertices.
Second, we show that typically there exists an $\ell$-cycle in $\mathbf{G}$ which cannot be activated by a `fast' process.

Before stating the next lemmas, let us make the following observation.
Let $\mathbf{A}=(H;\Delta_1,\ldots,\Delta_s)$
be an activation process with $|V(K(\mathbf{A}))|=v$. By the definition of $\exc_{\mathbf{A}}$ and $\con_{\mathbf{A}}$, we get
\begin{equation*}
      2d_H+\exc(\mathbf{A})\leq 2|E(K(\mathbf{A}))|+\sum_{e\in E(K(\mathbf{A}))}\con_\mathbf{A}(e) = 2|E(K(\mathbf{A}))|+2s < 4|E(K(\mathbf{A}))|.
\end{equation*}
Therefore, $2d_H+\exc(\mathbf{A}) < 4|E(K(\mathbf{A}))|\leq 2v^2$, implying that 
\begin{equation}\label{eq:Pvw=0}
    P_{v,w}=\emptyset \quad \text{for} \quad w>2v^2.
\end{equation}
We conclude that $Z^{\ell}_{v,w} \equiv 0$ deterministically, whenever $w>2v^2$. 

\begin{lemma}\label{lem: no ell-cycle between log and Klog}
    Fix constants $K_2>K_1> -\frac{1}{\ln(2^7c^3)}$ and $0<\eps<1/18$. 
     Then, w.h.p.\ for every integer $3\leq \ell=\ell(n)\leq n^{\eps}$, there are no $H,C\subseteq \mathbf{G}$ such that $C$ is an $\ell$-cycle, $|V(H)|\in [K_1\ell\ln n , K_2\ell\ln n]$, and $H \stackrel{\mathbf{G}}\longrightarrow C$ via a nice activation process.
\end{lemma}

\begin{proof}
Recall that an activation process of an $\ell$-cycle $\mathbf{A}=(H;\Delta_1,\ldots ,\Delta_s)$ defines an activation diagram.
    If further $\mathbf{A}$ is a nice process, with $|V(K(\mathbf{A}))|=v$, and with $2d_H+\exc({\mathbf{A}})=w$, then the activation diagram belongs to $P_{v,w}$.
    Thus, in order to conclude the proof, it suffices to show that w.h.p.\ for any $K_1\ell\ln n\leq v\leq K_2\ell\ln n$ and any $w\leq 2v^2$ there is no $D\in P_{v,w}$ such that $D\subseteq X^{(2)}(\mathbf{G})$.
We will show that, for all $n$ large enough and for every $3\leq\ell\leq n^{\varepsilon}$, with probability at least $1-n^{-1/2}$,
$$
Z_{\ell}\coloneqq \sum_{v=K_1\ell\ln n}^{K_2\ell\ln n}\sum_{w=0}^{2v^2} Z_{v,w}^{\ell}=0.
$$
This clearly suffices to finish the proof by the union bound.
As $Z_\ell$ is an integer-valued random variable, by Markov's inequality, it is enough to show that $\E[Z_{\ell}]\leq n^{-1/2}$ for all $\ell$ in the range. Indeed, we first write $\E[Z_\ell]$ in the following way:
\begin{align*}
	\E[Z_\ell]=  \sum_{v=K_1\ell\ln n}^{K_2\ell\ln n}\sum_{w=0}^{2v^2} \E\left[Z_{v,w}^{\ell}\right].
\end{align*}
By applying Lemma \ref{lem: expectation of the number of activations} in conjunction with our assumptions that $\eps<1/18$, $\ell\leq n^{\eps}$, and that $v\leq K_2\ell \ln n$, we obtain the following inequality: 
\begin{align*}
        \sum_{v=K_1\ell\ln n}^{K_2\ell\ln n}\sum_{w=0}^{2v^2} \E\left[Z_{v,w}^{\ell}\right]&\leq \sum_{v=K_1\ell\ln n}^{K_2\ell\ln n} \sum_{w=0}^{2v^2} n^{(-1/3+6\eps+o(1))w+(1/3+o(1))\ell+1} \cdot(2^7 c^3)^{v}\\
            &\leq \exp((1+\ell/3+o(\ell))\ln n) \cdot(2^7 c^3)^{K_1\ell \ln n}.
    \end{align*}


By our assumption that $K_1>-\frac{1}{\ln(2^7c^3)}$ we obtain the required inequality for large enough $n$:
\[
    \E[Z_\ell] \leq \exp((1+\ell/2)\ln n) \cdot(2^7 c^3)^{K_1\ell \ln n}<\exp((1-\ell/2)\ln n)\leq n^{-1/2}, 
\]
completing the proof.
\end{proof}

\begin{lemma}\label{lem: exists an ell-cycle with more than log}
    Fix positive constants $K$, $\eps<1/18$, and $\eps'<1/12$ . Then, for any $p \geq n^{-1/3-\eps'}$ w.h.p.\footnote{Even though the statement of this lemma holds for smaller values of $p$ as well, the considered range is sufficient for our goals and restricting $p$ on it makes the proof of this lemma shorter. Therefore, we state it here in this weaker form.} for every integer $4\leq \ell=\ell(n)\leq n^{\eps}$ there exists an $\ell$-cycle $C\subseteq \mathbf{G}$ with the following property: 
    There is no $H\subseteq \mathbf{G}$ with 
    $|V(H)|\in [\ell, K\ell\ln n]$ that activates $C$ via a nice activation process in $\mathbf{G}$.
\end{lemma}

\begin{proof}
Let $X_{\ell}$ denote the random variable counting the number of $\ell$-cycles $C\subseteq \mathbf{G}$ that \textbf{do not} satisfy the assertion of the lemma. Formally, $X_{\ell}$ counts the number of $\ell$-cycles $C\subseteq \mathbf{G}$ such that:
\[
	\exists H \subseteq \mathbf{G} \text{ with } |V(H)|\in [\ell, K\ell\ln n] \text{ and satisfying }  H\stackrel{\mathbf{G}}{\longrightarrow} C \text{ via a nice activation process}.
\]
Also, set $Y_{\ell}$ to be the random variable counting the number of $\ell$-cycles in $\mathbf{G}$. Recalling that $\ell\leq n^{\varepsilon}<n^{1/18}$ and $p\geq n^{-1/3-\varepsilon'}$, we get
\begin{equation}\label{eq:Y_ell_expectation}
\E[Y_{\ell}]=\frac{1}{2\ell}\cdot\frac{n!}{(n-\ell)!}\cdot p^{\ell}= \frac{1+o(1)}{2\ell}(np)^{\ell}\geq\frac{1+o(1)}{2\ell}(n^{2/3-\varepsilon'})^{\ell}=\omega(1).
\end{equation}
It is a standard exercise to check that $\mathrm{Var}(Y_{\ell})=O\left(\ell^3(\mathbb{E}[Y_{\ell}])^2/(n^2p)\right)$. Indeed, fix any $\ell$-cycle $C$ on an $\ell$-subset of $[n]$. Then, as $\ell^2=o(np)$ for all $\ell$ and $p$ in the considered ranges,
\begin{align*}
 \mathbb{E}[Y_{\ell}^2] &=\mathbb{E}[Y_{\ell}]\cdot\mathbb{E}[Y_{\ell}\mid C\subseteq \mathbf{G}]\leq\mathbb{E}[Y_{\ell}]\left(1+\sum_{i=1}^{\ell-2}{\ell\choose i}^2 n^{\ell-i-1}p^{\ell-i}\right)\\
 &\leq
 \mathbb{E}[Y_{\ell}]\left(1+\frac{1}{n}\sum_{i=1}^{\ell-2}\left(\frac{e^2\ell^2}{i^2 np}\right)^{i} (np)^{\ell}\right)=
 \mathbb{E}[Y_{\ell}]\left(1+O\left(\frac{\ell^2}{n^2p}(np)^{\ell}\right)\right)\\
 &=\mathbb{E}[Y_{\ell}]\left(1+O\left(\frac{\ell^3}{n^2p}\mathbb{E}[Y_{\ell}]\right)\right).
\end{align*}
Therefore, by the union bound and Chebyshev's inequality, \whp\ $Y_{\ell}\geq\mathbb{E}[Y_{\ell}]/2$, for every $\ell\leq n^{\eps}$.

In order to conclude the proof of the lemma, it is enough to show that w.h.p., for every $\ell\leq n^{\eps}$, 
\[
	X_{\ell}\leq \E[Y_{\ell}]/4.
\]
 In what follows we prove that for some small constant $\delta>0$
 $$
 \E[X_{\ell}] \leq  n^{(2/3-\eps'- \delta)\ell}\stackrel{\eqref{eq:Y_ell_expectation}}\leq (2+o(1))\ell n^{-\delta\ell}\cdot \E[Y_{\ell}].
 $$
 Then, Markov's inequality and the union bound over $\ell$ yield the desired assertion.

The rest of the proof resembles the proof of Lemma \ref{lem: no ell-cycle between log and Klog}. 
Recall that an activation process $\mathbf{A}=(H;\Delta_1,\ldots ,\Delta_s)$ of an $\ell$-cycle defines an activation diagram.
    If $\mathbf{A}$ is nice, $|V(K(\mathbf{A}))|=v$, and $2d_H+\exc({\mathbf{A}})=w$, then the activation diagram belongs to $P_{v,w}$.
    Thus, 
$$
\E[X_{\ell}]\leq \sum_{v= \ell}^{K\ell\ln n}\sum_{w=0}^{2v^2} \E[Z_{v,w}^{\ell}].
$$
By applying Lemma \ref{lem: expectation of the number of activations} in conjunction with our assumptions that $\eps < 1/18$, $c<2^{-7/3}$, $\ell\leq n^{\eps}$, and that $v\leq K\ell \ln n$, we obtain the following inequality:
\begin{align*}
         \sum_{v= \ell}^{K\ell\ln n}\sum_{w=0}^{2v^2} \E[Z_{v,w}^{\ell}]&\leq \sum_{v= \ell}^{K\ell\ln n}\sum_{w=0}^{v} n^{(-1/3+6\eps+o(1))w+(1/3+o(1))\ell+1}\leq n^{1+(1/3+o(1))\ell}.
\end{align*}
Set $\delta \coloneqq \frac{1}{2}\left(\frac{1}{12}-\eps'\right)>0$. Then, for $\ell\geq 4$, we have $1+\ell/3<\ell(2/3-\eps'-\delta)$, implying
\[
     \E[X_{\ell}] \leq n^{1+(1/3+o(1))\ell}\leq n^{(2/3-\eps'-\delta)\ell},
\]
and completing the proof.

\end{proof}

\subsection{Subprocesses}\label{subsection:3}

In this subsection, we introduce the notion of a `subprocess', and define the union of two subprocesses. Then, we connect these notions to the notion of activation complex from Subsection \ref{subsection:1} and study the contractibility of cycles in activation complexes of subprocesses and their unions.
We begin with the definition of a subprocess.

\begin{definition}
    Let $H$ be a graph, and let $\mathbf{A}=(H;\Delta_1,\ldots ,\Delta_s)$ be an activation process.
    For every graph $H'\subseteq K^{(1)}(\mathbf{A})$ and an activation process $\mathbf{B}=(H';\Delta'_1,\Delta'_2,\ldots, \Delta'_t)$, we say that $\mathbf{B}$ is a \emph{subprocess of $\mathbf{A}$} and write $\mathbf{B}\subseteq \mathbf{A}$ if the following holds:
\begin{enumerate}
	\item The sequence $(\Delta_k')_{k=1}^{t}$ is a subsequence of $(\Delta_k)_{k=1}^{s}$.
	\item For every $k$ the triangle $\Delta_k'$ activates\footnote{Recall that $\Delta_k$ activates an edge $e$ in the process $\mathbf{A}=(H;\Delta_1,\ldots ,\Delta_s)$ if $e\in \Delta_k\setminus (H\cup_{i=1}^{k-1}\Delta_i)$. } the same edge in both $\mathbf{A}$ and $\mathbf{B}$.

\end{enumerate}
\end{definition}


A key advantage of the above definition is that it allows to combine processes, when they are subprocesses of the same `ground' process.
Indeed, given an activation process $\mathbf{A}=(H;\Delta_1,\dots,\Delta_s)$ and subprocesses $\mathbf{A}_1=(H_1;\Delta_1^{(1)},\dots,\Delta_{s_1}^{(1)})$ and  $\mathbf{A}_2=(H_2;\Delta_1^{(2)},\dots,\Delta_{s_2}^{(2)})$, we define $\mathbf{A}_1\cup\mathbf{A}_2$ as follows.
Let $(\Delta_{i_j})_{j=1}^{t}$ be the subsequence of $(\Delta_{i})_{i=1}^{s}$ satisfying the following,
\begin{equation}\label{eq: union of activations}
	\{\Delta_{i_j}:j=1,\ldots,t\}=\{\Delta_i^{(1)}:i=1,\ldots ,s_1\}\cup \{\Delta_i^{(2)}:i=1,\ldots ,s_2\}.
\end{equation}

For every $k\in[t]$, we denote by $e_k$ the edge activated by $\Delta_{i_k}$ in $\mathbf{A}$. In a similar way, for every $i\in[s_1]$ and $j\in[s_2]$, we denote by $e^{(1)}_i$ and $e^{(2)}_j$ the edges activated by $\Delta^{(1)}_i$ in $\mathbf{A}_1$ and $\Delta^{(2)}_j$ in $\mathbf{A}_2$, respectively. 
Sequentially, for $k=0,1,\ldots,t$, we build graphs $H^{(k)}\subseteq H_1\cup H_2$ such that $(H^{(k)};\Delta_{i_1},\ldots , \Delta_{i_k})$ is an activation process. 
 We initially set $H^{(0)}=\emptyset$. 
Suppose $0\leq k<t$ and suppose that $H^{(k)}\subseteq H_1\cup H_2$ has been already defined such that $(H^{(k)};\Delta_{i_1},\ldots , \Delta_{i_k})$ is an activation process.
 If $(H^{(k)};\Delta_{i_1},\ldots,\Delta_{i_k},\Delta_{i_{k+1}})$ is an activation process as well, define $H^{(k+1)}\coloneqq H^{(k)}$. Clearly, then $(H^{k+1};\Delta_{i_1},\ldots , \Delta_{i_{k+1}})$ is an activation process.
Otherwise, define 
$$
H^{(k+1)}\coloneqq H^{(k)}\cup (\Delta_{i_{k+1}}\setminus (e_1\cup\ldots\cup e_k\cup e_{k+1})),
$$
which implies that $(H^{(k+1)};\Delta_{i_1},\ldots , \Delta_{i_{k+1}})$ is an activation process. It remains to show that $\Delta_{i_{k+1}}\setminus (e_1\cup\ldots\cup e_{k+1})\subseteq H_{1}\cup H_{2}$. 
Indeed, $\Delta_{i_{k+1}}$ appears in $\mathbf{A}_1$ or in $\mathbf{A}_2$. Without loss of generality assume it appears in $\mathbf{A}_1$, that is $\Delta_{i_{k+1}}=\Delta^{(1)}_{k'}$ for some $k'\in[s_1]$. 
Clearly, $(H_1;\Delta_1^{(1)},\ldots ,\Delta_{k'}^{(1)})$ is a subprocess of $(H^{(k+1)};\Delta_{i_1},\ldots,\Delta_{i_{k+1}})$. In particular, $\Delta^{(1)}_{k'}\setminus (e^{(1)}_1\cup\ldots\cup e^{(1)}_{k'})\subseteq H_{1}$. Therefore,
$$
 \Delta_{i_{k+1}}\setminus (e_1\cup\ldots\cup e_{k+1})=\Delta_{k'}^{(1)}\setminus (e_1\cup\ldots\cup e_{k+1})\subseteq\Delta_{k'}^{(1)}\setminus (e^{(1)}_1\cup\ldots\cup e^{(1)}_{k'})\subseteq H_{1}\subseteq H_{1}\cup H_{2},
$$
as required.


Finally, we define $\mathbf{A}_1\cup \mathbf{A}_2 \coloneqq (H_t; \Delta_{i_1},\ldots , \Delta_{i_{t}})$, which is clearly a subprocess of $\mathbf{A}$. Equality \eqref{eq: union of activations} implies the following trivial, yet important observation.

\begin{observation}\label{obs: union of complexes}
    For subprocesses $\mathbf{A}_1,\mathbf{A}_2\subseteq \mathbf{A}$ we have $K(\mathbf{A}_1\cup \mathbf{A}_2)=K(\mathbf{A}_1)\cup K(\mathbf{A}_2)$.
\end{observation}

From this observation, it follows that $\mathbf{A}_1$ and $\mathbf{A}_2$ are subprocesses of $\mathbf{A}_1\cup\mathbf{A}_2$.
In the next claim, we show a relation between contractibility of cycles and the union operator, that we have just defined.

\begin{lemma}\label{lm:union_contractible}
Suppose that $\mathbf{A}$ is an activation process and that $\mathbf{A}_1,\mathbf{A}_2\subseteq \mathbf{A}$ are subprocesses activating cycles $C_1$ and $C_2$, respectively. Moreover, suppose that $C_1\cap C_2$ is a path, and that $C_1,C_2$ are contractible in $K(\mathbf{A_1}),K(\mathbf{A_2})$, respectively. Then, the cycle $C_1\triangle C_2$ is contractible in $K(\mathbf{A}_1\cup \mathbf{A}_2)$.
\end{lemma}

\begin{proof}
By Definition~\ref{def:contractible_cycle}, there are simplicial fillings $D_1,D_2$ of $C_1,C_2$ in $K(\mathbf{A}_1),K(\mathbf{A}_2)$, respectively. By gluing the path $C_1\cap C_2$ on the boundary of $D_1$ to its copy on the boundary of $D_2$, we obtain a filling of $C_1\triangle C_2$ in $K(\mathbf{A}_1)\cup K(\mathbf{A}_2)$. By Observation \ref{obs: union of complexes} this is also a filling in $K(\mathbf{A}_1\cup\mathbf{A}_2)$, as required.
\end{proof}

We now recall that in Remark~\ref{rk:nice-minimal} we observed that nice activation processes are in a sense minimal. In the next lemma we make this intuition rigorous.

\begin{lemma}\label{lemma: activations contain good activations}
Suppose $\mathbf{A}$ is an activation process of a cycle $C$ such that $C$ is contractible in $K(\mathbf{A})$. 
Then, there is a \emph{nice} subprocess $\mathbf{A^*}\subseteq \mathbf{A}$ activating $C$, starting from a connected graph, and such that $C$ is contractible in $K(\mathbf{A^*})$.
\end{lemma}

\begin{proof}
We assume that $K^{(1)}(\mathbf{A})$ is connected. Otherwise, we may replace $\mathbf{A}$ with its subprocess activating the connected component of $K^{(1)}(\mathbf{A})$ containing $C$. Clearly, this preserves contractibility~of~$C$.

We now introduce an iterative process which refines $\mathbf{A}$. Formally, we define a sequence $\mathbf{A}\eqqcolon\mathbf{A}_0\supseteq \mathbf{A}_1\supseteq \ldots $ of activation processes of $C$. Throughout this iterative definition, we make sure that, for all $k$, the graph $K^{(1)}(\mathbf{A}_k)$ is connected and that  $C$ is contractible in $K(\mathbf{A}_k)$. 


Let $k\geq 1$ and assume that the processes $\mathbf{A}_0\supseteq\ldots\supseteq \mathbf{A}_{k-1}$ are defined. If $\mathbf{A}_{k-1}=:(H;\Delta_1,\ldots,\Delta_s)$ is a nice process, then halt. 
 Otherwise, there is an edge $e\in K^{(1)}(\mathbf{A}_{k-1})$ with a negative excess. 
 First, if $e\in E(H)\cap E(C)$ then $\con_{\mathbf{A}_{k-1}}(e)\geq 1$, which is not possible as $\exc_{\mathbf{A}_{k-1}}(e)<0$. Indeed, $C$ is contractible in $K(\mathbf{A}_{k-1})$, and thus has a simplicial filling in $K(\mathbf{A}_{k-1})$. Since $e\in E(C)$, there is a face in $K(\mathbf{A}_{k-1})$ that $e$ belongs to. This face participates in the activation process and, since $e\in E(H)$, it is used to activate an edge other than $e$. This means that $\con_{\mathbf{A}}(e)\geq 1$. Second, if $e\in E(C)\setminus E(H)$, then $\exc_{\mathbf{A}_{k-1}}(e)=\con_{\mathbf{A}_{k-1}}(e)\geq 0$, which is again impossible. Hence, $e\not\in C$, and in particular, there are two options: either $e\in H\setminus C$ or $e\not\in H\cup C$.
We consider these two cases separately.


\paragraph{Case 1.} Assume $e\not\in H\cup C$. Since $\exc_{\mathbf{A}_{k-1}}(e)<0$, we have $\con_{\mathbf{A_{k-1}}}(e)=0$. Let $\Delta_i$ be the triangle activating $e$, and set
\[
    \mathbf{A}_k\coloneqq (H;\Delta_{1},\ldots,\Delta_{i-1},\Delta_{i+1},\ldots,\Delta_s).
\]
Note that $K^{(1)}(\mathbf{A}_k)$ is connected as $H$ is its connected spanning subgraph.
Let us now show that $C$ is contractible in $K( \mathbf{A}_k)$. Indeed, $C$ is contractible in $K( \mathbf{A}_{k-1})$, and thus there is a simplicial filling $D_{k-1}$ of $C$ in $K(\mathbf{A}_{k-1})$. 
Since the only face in $K(\mathbf{A}_{k-1})$ containing $e$ is $\Delta_{i}$, every triangle containing a preimage of $e$ in $D_{k-1}$ is a preimage of $\Delta_{i}$. Moreover, since $e\not \in C$, every preimage of $e$ in $D_{k-1}$ is contained in two preimages of $\Delta_{i}$.
Iteratively, for every preimage $\{u,v\}$ of $e$ that belongs to two faces $\{u,v,z_1\}$ and $\{u,v,z_2\}$ in $D_{k-1}$, remove the edge $\{u,v\}$, and identify $z_1=z_2$ by gluing edges $\{z_1,u\}$ and $\{z_1,v\}$ with $\{z_2,u\}$ and $\{z_2,v\}$, respectively; see Figure \ref{Fig-collapse}. 
 Since $D_{k-1}$ is finite, the process terminates, and we obtain the triangulated disc $D_k$ which is a filling of $C$ in $K(\mathbf{A}_{k-1})$.
Since $D_{k}$ does not contain $\Delta_{i}$, it is also a filling of $C$ in $K(\mathbf{A}_{k})$, as required.

\begin{figure}[ht]
\centering
\includegraphics[scale=0.49]{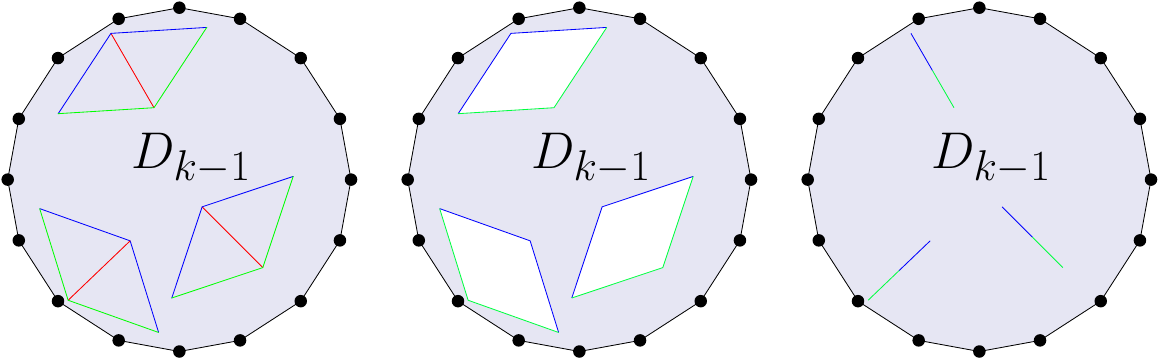}
  \caption{The rainbow triangles are preimages of $\Delta_i$, where the preimage of $e$ is coloured red.}\label{Fig-collapse}
\end{figure}

\paragraph{Case 2.} Assume $e\in {H}\setminus C$. Since $\exc_{\mathbf{A}_{k-1}}e<0$, there are two possibilities: either $\con_{\mathbf{A}_{k-1}}(e)=0$, or $\con_{\mathbf{A}_{k-1}}(e)=1$.

First, suppose that $\con_{\mathbf{A}_{k-1}}(e)=0$. Since $e\not\in C$, we get that $C$ is being activated also by 
\[ 
	 \mathbf{A}'\coloneqq (H\setminus e;\Delta_{1}, \ldots, \Delta_{s}), 
\]
and, moreover, $C$ is contractible in $K(\mathbf{A}')$ --- as $K^{(2)}(\mathbf{A}')\setminus  K^{(1)}(\mathbf{A}')=K^{(2)}(\mathbf{A})\setminus K^{(1)}(\mathbf{A})$. Let $\mathbf{A}_k\subseteq \mathbf{A'}$ be the (unique) subprocess such that $K(\mathbf{A}_k)$ is exactly the connected component of $K(\mathbf{A'})$ containing $C$. Then $C$ is also contractible in $K(\mathbf{A}_k)$. Moreover, since $K^{(1)}(\mathbf{A}_k)$ is connected, the process $\mathbf{A}_k$ is initiated by a connected spanning subgraph of $K^{(1)}(\mathbf{A}_k)$. 

It remains to consider the case $\con_{\mathbf{A}_{k-1}}(e)=1$. Let $\Delta_{i}$ be the unique triangle of $\mathbf{A}_{k-1}$ containing $e$. Also, let $e'$ be the edge activated by $\Delta_{i}$ in the process $\mathbf{A}_{k-1}$. Define
\[
     \mathbf{A}_k\coloneqq (H\cup e' \setminus e;\Delta_{1},\ldots,\Delta_{i-1},\Delta_{i+1},\ldots,\Delta_{s}),
\]
see Figure \ref{Fig-changing edges}. Let us now show that $H\cup e' \setminus e$ is connected, and that $C$ is contractible in $K(\mathbf{A}_k)$.

Assume towards contradiction that $H\cup e' \setminus e$ is not connected. In particular, $e$ is a bridge in $H$, and $H\setminus e$ has two connected components $J_1$ and $J_2$. Since $\Delta_i$ is the only triangle that contains $e$, we get that $\Delta_i$ is the first triangle in the process $\mathbf{A}_{k-1}$ that activates an edge with one endpoint in $J_1$ and the other endpoint in $J_2$. Indeed, the first triangle that activates such an edge has to have another edge between $J_1$ and $J_2$, and then the latter edge must be in $H$ --- the only such edge is $e$. 
So, the edge $e'$ that $\Delta_i$ activates has one endpoint in $J_1$ and the other endpoint in $J_2$, implying that $H\cup e' \setminus e$ is actually connected --- a contradiction.

Finally, let us show that $C$ is contractible in $K(\mathbf{A}_k)$. Since $C$ is contractible in $K(\mathbf{A}_{k-1})$, there is a simplicial filling $D_{k-1}$ of $C$ in $K(\mathbf{A}_{k-1})$. Since, in this case, the only triangle in $K(\mathbf{A}_{k-1})$ containing $e$ is $\Delta_{i}$, every triangle containing a preimage of $e$ in $D_{k-1}$ is a preimage of $\Delta_{i}$. We then argue in the same way as in Case 1: Iteratively delete each preimage of $e$ and glue the identical edges that remain after the deletion. Eventually, we get a simplicial filling $D_k$ of $C$ in $K(\mathbf{A}_k)$, as required.\\


\begin{figure}[ht]
    \centering
    \includegraphics[scale=0.49]{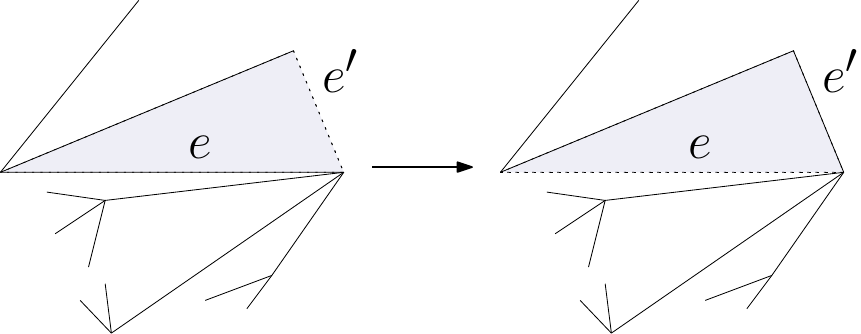}
    \caption{Replacing $e$ with $e'$ in the highlighted triangle $\Delta_i$.}
    \label{Fig-changing edges}
\end{figure}

The desired sequence of nested processes is constructed. Since for every $k\geq 1$, we have $K(\mathbf{A}_{k-1})\supsetneq K(\mathbf{A}_{k})$, and since $K(\mathbf{A}_{0})$ is a finite complex, the iterative process terminates. The resulting subprocess of $\mathbf{A}$ is then the required nice activation process.
\end{proof}

\subsection{From local to global}
\label{subsection:4}

This section introduces the final step of the proof, inspired by Gromov's `from local to global' principle. Informally, we show that for a graph $G$, if every fairly short cycle cannot be activated by a `moderate' process in $G$, and moreover, there is a fairly short cycle which a `fast' activation process cannot activate, then this cycle cannot be activated at all. This clearly means that $G$ itself cannot be activated. 

Using lemmas from Section~\ref{subsection:2}~and~Section~\ref{subsection:3}, we show here that under the assumption of the second part of Theorem \ref{th:shallow}, \whp\ $G_{n,p}$ satisfies the above conditions. This and the main result of this subsection will then conclude the proof of the second part of Theorem \ref{th:shallow}.

The above is given formally in the following proposition.

\begin{proposition}[From local to global]\label{prop: local to global}  
    Let $n,L$ and $K_2>K>100K_1$ be positive integers. Further, let $G$ be an $n$-vertex graph satisfying the following \emph{local} properties:
    \begin{enumerate}[label=\textbf{(L\arabic*)}]
        \item\label{local1} For every $3\leq \ell\leq L$ there are no subgraphs $H,C\subseteq G$ such that $C$ is an $\ell$-cycle, $|V(H)|\in [K_1\ell \ln n,K_2\ell \ln n]$, and $H\stackrel{G}{\longrightarrow}C$ via a nice activation process. 
        \item\label{local2} For every $4\leq \ell\leq L$ there is an $\ell$-cycle $C\subseteq G$ such that no $H\subseteq G$ with $|V(H)|\in [\ell,K\ell \ln n]$ activates $C$ in $G$  via a nice activation process. 
    \end{enumerate}
    Then, there is no tree $T\subseteq G$ with diameter at most ${L/100}$ activating $G$.
\end{proposition}

We first derive the second part of Theorem \ref{th:shallow} from Proposition \ref{prop: local to global} and then prove Proposition~\ref{prop: local to global} itself.

\begin{proof}[Proof of Theorem \ref{th:shallow}, part 2]

Fix any positive constants $\varepsilon<1/18$, $\eps'<1/12$, $c<2^{-7/3}$, and let $\mathbf{G}\sim G_{n,p}$, where 
$$
 n^{-1/3-\varepsilon'}\leq p<cn^{-1/3}.
$$ 
Let us also fix 
$$
 K_2>K>100K_1\geq -\frac{1}{\ln(2^7 c^3)}.
$$
The assertions of Lemma \ref{lem: no ell-cycle between log and Klog}~and~Lemma \ref{lem: exists an ell-cycle with more than log} hold for these parameters. Therefore, \whp\ $\mathbf{G}$ satisfies the conclusions of these lemmas.

Note that by setting $L\coloneqq n^{\eps}$, the conclusions of Lemma \ref{lem: no ell-cycle between log and Klog}~and~Lemma \ref{lem: exists an ell-cycle with more than log} are exactly \ref{local1} and \ref{local2}, respectively. Thus, due to Proposition \ref{prop: local to global} \whp\ $\mathbf{G}$ cannot be activated by any tree of diameter at most $n^{\eps}/100$, as required.
\end{proof}

The rest of this subsection is devoted to the proof of Proposition \ref{prop: local to global}.

\begin{proof}[Proof of Proposition \ref{prop: local to global}]

Let $G$ be as in the proposition. Assume, towards contradiction, that $G$ can be activated from a spanning subtree $T\subseteq G$ with $\diam(T) =:  d \leq L/100 $. Let $\mathbf{T}\coloneqq (T;\Delta_1,\Delta_2,\ldots ,\Delta_s)$ be an activation process of $G$.
 
Next, we define the set of all cycles that can be activated in ${G}$ via a fast subprocess of $\mathbf{T}$. Let $\mathcal{X}_{\ell}$ be the set of all $\ell$-cycles $C\subseteq{G}$ such that there exists a connected subgraph $H\subseteq{G}$ that activates $C$ via a {\it nice} subprocess $\mathbf{A}\subseteq\mathbf{T}$ satisfying the following two properties: 
\begin{itemize}
    \item
     $|V(H)|\leq K_2\ell\ln n$;
    \item
     $C$ is contractible in $K(\mathbf{A})$. 
\end{itemize}
Set $\mathcal{X}:= \cup_{3\leq \ell \leq 100d}\mathcal{X}_{\ell}$.    

Note that, by \ref{local1}, any graph $H$ that activates $C\in\mathcal{X}$ via a nice process as above, actually has $|V(H)|\leq K_1\ell\ln n$. Moreover, by \ref{local2}, for every $\ell\in[4,100d]$, there exists an $\ell$-cycle $C$ such that every nice process $\mathbf{A}$ activating it has $|K^{(0)}(\mathbf{A})|>K\ell\ln n>K_1\ell\ln n$, implying $C\notin\mathcal{X}_{\ell}$. We get that for all $\ell\in[4,100d]$, there is an $\ell$-cycle $C\subseteq G$, that does not belong to $\mathcal{X}$.



For every cycle $C$ that has length at most $100d$, let $s(C)$ be the minimum number of triangles in a subprocess $\mathbf{A}\subseteq\mathbf{T}$ that starts from $T$ and activates $C$. Since the activation process of such $C$ starts from $T$, we get that $C$ is contractible in $K(\mathbf{A})$, due to Lemma~\ref{lem: main lemma for 0-statement}. We then fix a cycle $C\notin\mathcal{X}$ that has length $\ell:=|V(C)|\in[3,100d]$ with the minimum $s(C)$ among all such cycles. Let $\mathbf{A}=(T;\Delta_1,\ldots,\Delta_s)\subseteq\mathbf{T}$ be a subprocess of length $s=s(C)$ that activates $C$. Let $e$ be the last edge of $C$ activated by $\mathbf{A}$. If $s=1$, then, since $T$ is acyclic, $C=\Delta_1$. We get that $(\Delta_1\setminus e;\Delta_1)\subset\mathbf{T}$ is a nice subprocess activating $C$ and $C$ is contractible in its activation complex, implying $C\in\mathcal{X}$ --- a contradiction. Thus, $s\geq 2$.



In the next two claims, we describe important properties of $\mathbf{A}$ that we will then use to complete the proof of Proposition~\ref{prop: local to global}. Recall that the triangle $\Delta_s$ activates the edge $e$. Let $e',e''$ be the other two edges of $\Delta_s$. Moreover, let $C^*$ be the graph obtained from $C$ by deleting the edge $e$ and adding the edges $e'$ and $e''$.

\begin{claim}\label{claim: last triangle is not boundary triangle}
   We have $|V(e'\cap C)|=|V(e''\cap C)|=1$, that is both $e'$ and $e''$ share only one vertex with $C$.
\end{claim}

\begin{proof}
Both $e',e''$ share at least one vertex with $C$, since $e$ is an edge of $C$. Then, we should only exclude the situation when $|V(e'\cap C)|=|V(e''\cap C)|=2$, which happens if the common vertex $z$ of $e'$ and $e''$ lies on $C$ --- see Figure~\ref{Fig-last triangle in C^*}.

    
Assume towards contradiction that $z\in V(C)$. Without loss of generality, we may assume that $e'\not \in E(C)$, as $s\geq 2$. Note that under this assumption, either the graph $C^*$ consists of two cycles $C_1$ and $C_2$ with the common vertex $z$ (see Figure \ref{Fig-the two cycles}), or $C^*$ consists of a cycle $C_1$ and the edge $e''=:C_2$ that shares $z$ with $C_1$. In both cases, $\max\{V(C_1),V(C_2)\}\leq 100d$.

\begin{figure}[ht]
\centering

\begin{minipage}{0.5\textwidth}
  \centering
  \includegraphics[scale=0.49]{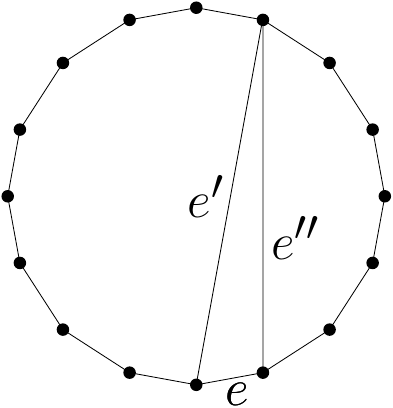}
  \caption{The case $|V(e'\cap C)|=|V(e''\cap C)|=2$}\label{Fig-last triangle in C^*}
\end{minipage}%
\begin{minipage}{.5\textwidth}
  \centering
    \includegraphics[scale=0.49]{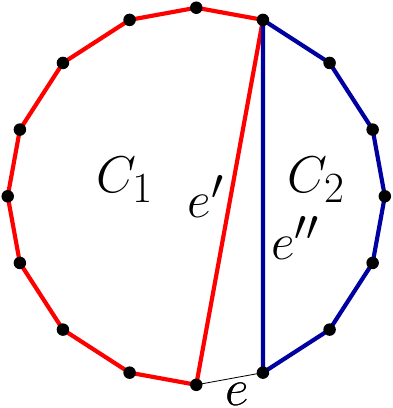}
  \caption{The two cycles $C_1$ and $C_2$ in $C^*$}\label{Fig-the two cycles}
\end{minipage}

\end{figure}

Both $C_1$ and $C_2$ can be activated by a subprocess of $\mathbf{A}$ without the last triangle $\Delta_s$. Therefore, by the minimality of $s(C)$, these cycles --- unless $C_2=e''$ --- are in $\mathcal{X}$. Then, both $C_1$ and $C_2$ can be activated via nice subprocesses $\mathbf{A}_1,\mathbf{A}_2\subseteq \mathbf{T}$ such that $\max\{|V(K(\mathbf{A}_1))|,|V(K(\mathbf{A}_2))|\}\leq K_1\ell\ln n$ and $C_1,C_2$ are contractible in $K(\mathbf{A}_1),K(\mathbf{A}_2)$, respectively. If $C_2=e''$, we let $K(\mathbf{A}_2)=C_2$ (note that $\mathbf{A}_2\subseteq \mathbf{A}$ since clearly $e''\in K^{(1)}(\mathbf{A})$). Set $\mathbf{A}_3\coloneqq (\{e',e''\};\Delta_s)$. Observe that $\mathbf{A}_3\subseteq\mathbf{A}$ 
 is a nice activation process of $\Delta_s$ and $\Delta_s$ is contractible in $K(\mathbf{A}_3)$. We then apply Lemma~\ref{lm:union_contractible} twice and get that the process $\mathbf{A}^{\cup}:=\mathbf{A}_1\cup\mathbf{A}_2\cup\mathbf{A}_3$ activates $C$ and $C$ is contractible in $K(\mathbf{A}^{\cup})$. Therefore, by Lemma \ref{lemma: activations contain good activations} there is a nice activation subprocess of $\mathbf{A}^{\cup}$ that activates $C$ and such that $C$ is contractible in the respective activation complex. By Observation~\ref{obs: union of complexes}, $K(\mathbf{A}^{\cup})$ has at most $2K_1\ell\ln n<K_2\ell \ln n$ vertices which contradicts the assumption that $C\notin\mathcal{X}$ and concludes the proof of the claim.
\end{proof}

We then conclude that $C^*$ is an $(\ell+1)$-cycle.


\begin{claim}
    $\ell=100d$.
 \label{cl:ell=100d}
\end{claim}

\begin{proof}
Assume towards contradiction that $\ell<100d$. Claim \ref{claim: last triangle is not boundary triangle} implies that $C^*$ is a cycle of length $\ell+1\leq 100d$. It can be activated by the subprocess of $\mathbf{A}$ without the last triangle $\Delta_s$. Recalling that $\mathbf{A}$ starts from $T$ and, therefore, the considered subprocess as well, by the minimality of $s(C)$, we get that $C^*\in\mathcal{X}$. In particular, $C^*$ admits a nice activation subprocess of $\mathbf{A}^*\subseteq\mathbf{T}$ with less than $K_1(\ell+1)\ln n$ vertices, and such that $C^*$ is contractible in $K(\mathbf{A}^*)$. 

 As in the proof of Claim \ref{claim: last triangle is not boundary triangle}, we note that $\mathbf{A}^{**}\coloneqq (\{e',e''\};\Delta_s)$ is a nice activation process of $\Delta_s$ and $\Delta_s$ is contractible in $K(\mathbf{A}^{**})$. We also note that $\mathbf{A}^{**}$ is a subprocess of $\mathbf{T}$ since $\Delta_s$ participates in $\mathbf{T}$, where it activates the same edge as in $\mathbf{A}^{**}$. We then apply Lemma~\ref{lm:union_contractible} and get that the process $\mathbf{A}^{\cup}:=\mathbf{A}^*\cup\mathbf{A}^{**}$ activates $C$ and $C$ is contractible in $K(\mathbf{A}^{\cup})$. Therefore, by Lemma \ref{lemma: activations contain good activations} there is a nice activation subprocess of $\mathbf{A}^{\cup}$ that activates $C$ and such that $C$ is contractible in the respective activation complex. By Observation~\ref{obs: union of complexes}, $K(\mathbf{A}^{\cup})$ has at most $K_1(\ell+1)\ln n <K_2\ell \ln n$ vertices --- a contradiction to the assumption that $C\notin\mathcal{X}$.
\end{proof}

We are now ready to complete the proof of Proposition \ref{prop: local to global}. Below we describe the last step of the proof: roughly speaking, we show that if an $\ell$-cycle is activated via a process involving more than $K_2\ell\ln n$ triangles, then there is also an $\ell'$-cycle whose activation process essentially involves `moderately' many triangles, that is some number which is between $K_1\ell'\ln n$ and $K_2\ell'\ln n$. This would contradict \ref{local1} and complete the proof of Theorem~\ref{th:shallow} due to the observation that $\mathcal{X}\neq\emptyset$. 

Consider two antipodal vertices $x,y$ of $C^*$, that is $\dist_{C^*}(x,y)=50d$. Let  $P$ the unique path between $x$ and $y$ in $T$. Recall that $T$ has diameter at most $d$, and therefore $|E(P)|\leq d$. 
\begin{claim}\label{claim: creating two cycles}
    There is a path $P'\subseteq P$ such that
    \[
        V(P')\cap V(C^*)=\{a,b\} \quad \text{and}\quad\dist_{C^*}(a,b)>\dist_{P}(a,b)=|E(P')|.
    \]
\end{claim}

\begin{proof}
    Let $P=(x=v_1v_2\ldots v_r=y)$, and recall that $r\leq d+1$. Let $1=i_1<\ldots<i_k=r$ be indices of all the vertices $v_{i_1},\ldots,v_{i_k}$ of the path $P$ that also belong to the cycle $C^*$. Then, for all $1\leq j\leq k-1$ we let $P_{j}\subseteq P$ be the induced subpath between $v_{i_j}$ and $v_{i_{j+1}}$. Let us show that at least one of these $P_j$ satisfies the assertion of the claim. Indeed, otherwise, for all $1\leq j\leq k-1$ we have $\dist_{C^*}(v_{i_j},v_{i_{j+1}})\leq \dist_{P}(v_{i_j},v_{i_{j+1}})$, which by the triangular inequality implies
    \[
        50d=\dist_{C^*}(x,y)\leq \sum_{j=1}^{k-1}\dist_{C^*}(v_{i_j},v_{i_{j+1}})\leq \sum_{j=1}^{k-1}\dist_{P}(v_{i_j},v_{i_{j+1}})=\dist_P(x,y)=r-1 \leq d
    \]
--- a contradiction.
\end{proof}

Fix $P'$ as in Claim \ref{claim: creating two cycles}, and set $V(P')\cap V(C^*)=\{a,b\}$.
Let $P_1,P_2$ be the two distinct paths starting from $a$ and ending at $b$ in $C^*$. The path $P'$ splits $C^*$ into two cycles $C_1$ and $C_2$, where, for $i\in\{1,2\}$, $C_i=P_i\cup P'$ --- see Figure \ref{Fig-local-to-global}. Without loss of generality, assume that $|E(P_1)|\leq |E(P_2)|$. By Claim \ref{claim: creating two cycles}, we also have $|E(P')|< |E(P_1)|\leq |E(P_2)|$. In particular, it implies that both $C_i$ are actually cycles, i.e. $|E(C_2)|\geq |E(C_1)|\geq 3$ . Moreover, recalling that $|E(C^*)|=\ell+1=100d+1$ due to Claim~\ref{cl:ell=100d},  we get that, for $i\in\{1,2\}$,
$$
   |E(C_i)|=|E(P_i)|+|E(P')| \leq |E(P_i)|+|E(P_1)|-1\leq |E(P_2)|+|E(P_1)|-1=|E(C^*)|-1=100d.
$$
We conclude that $3\leq |E(C_i)|\leq 100d$ for both $i=1$ and $i=2$.

Both $C_1,C_2$ are activated from the subprocess of $\mathbf{A}$ without the last triangle $\Delta_s$, since this subprocess starts from $T$ and $P'\subseteq T$.
By the minimality of $s(C)$, we get that $C_1,C_2\in\mathcal{X}$. Hence, both cycles $C_1,C_2$ admit nice activation subprocesses $\mathbf{A}_1,\mathbf{A}_2\subseteq\mathbf{T}$ with at most $K_1\ln n$ vertices, and such that $C_1,C_2$ are contractible in $K(\mathbf{A}_1),K(\mathbf{A}_2)$, respectively. In the same way as in the proofs of Claims~\ref{claim: last triangle is not boundary triangle},~\ref{cl:ell=100d}, we now observe that $\mathbf{A}_3\coloneqq (\{e',e''\};\Delta_s)\subseteq\mathbf{T}$ is a nice activation process of $\Delta_s$ and $\Delta_s$ is contractible in $K(\mathbf{A}_3)$. Two applications of  Lemma~\ref{lm:union_contractible} imply that the process $\mathbf{A}^{\cup}:=\mathbf{A}_1\cup\mathbf{A}_2\cup\mathbf{A}_3$ activates $C$ and $C$ is contractible in $K(\mathbf{A}^{\cup})$. Therefore, by Lemma \ref{lemma: activations contain good activations} there is a nice activation subprocess of $\mathbf{A}^{\cup}$ that activates $C$ and such that $C$ is contractible in its activation complex. By Observation~\ref{obs: union of complexes}, $|K^{(0)}(\mathbf{A}^{\cup})|\leq 2K_1\ell\ln n<K_2\ell \ln n$. This contradicts the assumption that $C\notin\mathcal{X}$ and concludes the proof of Proposition \ref{prop: local to global}.\qedhere

\begin{figure}[ht]
\centering
\includegraphics[scale=0.49]{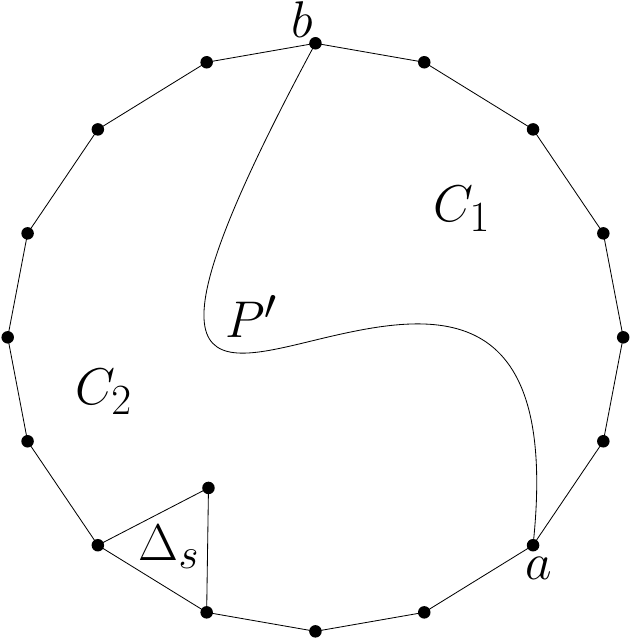}
  \caption{The three cycles $C_1$, $C_2$, and $\Delta_s$ which admit fast activation processes}\label{Fig-local-to-global}
\end{figure}

\end{proof}

\section{Concluding remarks}
\label{sc:discussions}

In this section, we discuss several generalisations of our result as well as remaining challenges. 

First, in Section~\ref{subsc:LM}, we show that our method is robust enough to tackle a similar problem in the Linial-Meshulam complexes. The problem of activating the 1-skeleton of the Linial-Meshulam complex from a tree was studied in \cite{KorPelSud}. 
We show that our methodology allows to strengthen the main result of \cite{KorPelSud}.
 
Second, in Section~\ref{subsc:hampath}, we show that, when $p\geq (1+\eps)n^{-1/3}$, the diameter of a tree that activates $G_{n,p}$ can be `artificially' extended: \whp\ there exists a Hamilton path $P$ in $G_{n,p}$ such that $P\to G_{n,p}$. On the one hand, it shows that it is not the case that the diameter of a tree in the first part of Theorem~\ref{th:general} has to be bounded. On the other hand, this is essentially the same tree --- there exists an activation process of $P$ from a certain tree $T$ with diameter 4 such that $P$ can be obtained from $T$ in the corresponding activation complex via a sequence of transformations as in Figure~\ref{Fig-changing edges}. We actually believe that any tree that has diameter more than $n^{1/18-\varepsilon}$ can be reduced (in the above sense) to a tree with diameter at most $n^{1/18-\varepsilon}$. If this is true, this would prove that the upper bound from Theorem~\ref{th:general} is tight, up to a constant factor.

Finally, in Section~\ref{subsc:questions}, we discuss remaining questions and future avenues.

\subsection{An analogue in the Linial-Meshulam complex}
\label{subsc:LM}


The aim of this subsection is to prove an analogue of Theorem~\ref{th:shallow} in the setting of the Linial-Meshulam complexes. Since the proof resembles the proof of Theorem~\ref{th:shallow}, we only sketch the former proof here and pay most attention to the steps in the latter proof that have to be modified. Nevertheless, for completeness, a full proof is given in Appendix \ref{Appendix}.


Let us start by introducing some notations. Let $n$ be an integer, and let $p\in (0,1)$. The ($2$-dimensional) {\it Linial-Meshulam complex} $H_{n,p}$ is constructed by setting $H_{n,p}^{(1)}\coloneqq K_{n}$ and then including every triangle, independently of the others, with probability $p$.

Let $F\subseteq G\subseteq  K_n$, and let $X$ be a $2$-dimensional complex with $X^{(1)}=K_n$. 
We say that $F$ {\it triadically activates} $G$ in $X$ if there is a sequence of graphs $F\coloneqq F_0 \subset F_1 \subset \cdots \subset F_s \eqqcolon G$ where $F_i$ is obtained from $F_{i-1}$ by adding an edge $e_i$ that belongs to a triangle $\Delta_i\in X$ whose edges, in turn, belong to $F_i$. In this case, we write $F\stackrel{X}{\rightarrow}G$ 
and say that {\it $F$ activates $G$}. Notice that this notion of activation differs from its analogue in graphs, defined in Section~\ref{sc:results}: here, we are only allowed to use triangles that belong to $X$, although a graph activated at either step may induce some new triangles that do not belong to $X$.


We prove the following analogue to the 0-statement in Theorem \ref{th:shallow}.

\begin{theorem}\label{th:0-statement LM}
    Let $\eps>0$ be a constant, let $0<p\coloneqq p(n)\leq (2^{-7/2}-\eps)n^{-1/2}$, and let $\mathbf{H}\sim H_{n,p}$.
    Then, w.h.p.\ there is no tree $T\subset K_n$ of diameter at most $n^{1/12-\eps}$ such that $T\stackrel{\mathbf{H}}{\longrightarrow} K_n$. 
\end{theorem}

 The $1$-statement was proved by Kor\'{a}ndi, Peled, and Sudakov \cite{KorPelSud}.  They also proved that the bound in the 1-statement is sharp for stars. 

\begin{theorem}[Kor\'{a}ndi, Peled, Sudakov \cite{KorPelSud}]\label{th:KorPelSud}
    Let $\varepsilon>0$, $p=p(n)\in [0,1]$, and $\mathbf{H}\sim H_{n,p}$.
    \begin{enumerate}
        \item If $p\geq (1/2+\eps)n^{-1/2}$, then \whp\ there exists a star $S\cong K_{1,n-1}$ such that $S\stackrel{\mathbf{H}}{\longrightarrow} K_n$.
        \item If $p\leq (1/2-\eps)n^{-1/2}$, then \whp\ there is no star $S\cong K_{1,n-1}$ such that $S\stackrel{\mathbf{H}}{\longrightarrow}K_n$.
    \end{enumerate}
\end{theorem}

Although the constant factor in the 0-statement in Theorem \ref{th:0-statement LM} does not match the constant factor in the 1-statement in Theorem~\ref{th:KorPelSud}, it establishes the activation threshold, up to a constant factor, for all trees that have diameter at most $n^{1/12-\eps}$.

Before sketching the proof of Theorem~\ref{th:0-statement LM}, let us notice that, when $p<n^{-1/2-\varepsilon}$, w.h.p. there is no tree $T\subset K_n$ such that $T\stackrel{\mathbf{H}}{\longrightarrow}K_n$, due to the celebrated result of Babson, Hoffman, and Kahle~\cite{BabHofKah2011} and Lemma \ref{lem: main lemma for 0-statement}.

\begin{theorem}[Babson, Hoffman, Kahle \cite{BabHofKah2011}]
 For every $\eps>0$ and $p<n^{-1/2-\eps}$ w.h.p.\ $\mathbf{H}$ is not simply connected.     
\end{theorem}


The proof of Theorem \ref{th:0-statement LM} is almost verbatim to the proof of the 0-statement in Theorem \ref{th:shallow}. Recall that the 0-statement in Theorem \ref{th:shallow} follows from Lemma \ref{lem: no ell-cycle between log and Klog}, Lemma \ref{lem: exists an ell-cycle with more than log}, and Proposition~\ref{prop: local to global}. Let us state the analogues of Lemmas~\ref{lem: no ell-cycle between log and Klog},~\ref{lem: exists an ell-cycle with more than log} that, combined with Proposition~\ref{prop: local to global}, imply Theorem~\ref{th:0-statement LM}.

 Let $\mathbf{H}\sim H_{n,p}$, where $p<cn^{-1/2}$ and $c\in(0,2^{-7/2})$ is a constant. Moreover, let $\ell \coloneqq \ell(n)\geq 3$ be a sequence of integers, and also let $v=v(n),w=w(n)$ be two sequences of non-negative integers. Finally, let $Z^{\ell}_{v,w}$ be the (random) number of activation diagrams $D\in P_{v,w}$ such that $D\subseteq \mathbf{H}$\footnote{As in Section~\ref{subsection:2}, it means that all triangles of $D$ also belong to $\mathbf{H}$.}. The proofs of the following two lemmas are presented in Appendix \ref{Appendix}.


\begin{restatable}{lemma}{xyz}\label{lem: no ell-cycle between log and Klog LM}
    Fix constants $K_2>K_1> -\frac{1}{\ln(2^7c^2)}$ and $0<\eps<1/12$.
     Then, w.h.p.\ for every integer $3\leq \ell=\ell(n)\leq n^{\eps}$, there are no $H,C\subseteq K_n$ such that $C$ is an $\ell$-cycle, $|V(H)|\in [K_1\ell\ln n , K_2\ell\ln n]$, and $H \stackrel{\mathbf{H}}\longrightarrow C$ via a nice activation process.
\end{restatable}

\begin{restatable}{lemma}{ab}\label{lem: exists an ell-cycle with more than log LM}
    Fix positive constants $K$, $\eps<1/12$. Then, w.h.p.\ for every integer $4\leq \ell=\ell(n)\leq n^{\eps}$ there exists an $\ell$-cycle $C\subseteq K_n$ with the following property: 
    There is no $H\subseteq K_n$ with 
    $|V(H)|\in [\ell, K\ell\ln n]$ that activates $C$ via a nice activation process in $\mathbf{H}$.
\end{restatable}

In the same way as in Section~\ref{subsection:4}, the proof of Theorem \ref{th:0-statement LM} follows by combining the above typical properties of $\mathbf{H}$ with the local-to-global Proposition \ref{prop: local to global}. 

\subsection{Activating from a Hamilton path}\label{subsc:hampath}

In this section, we prove that there is a witness for the $1$-statement in Theorem \ref{th:general}, which is a Hamilton path rather than from a depth $2$ tree.

Let $\eps>0$ be a fixed constant, $p:=p(n)\geq(1+\eps)n^{-1/3}$, and $\mathbf{G}\sim G_{n,p}$.
In what follows, we use notation from Section \ref{sc:proof_shallow_1}.
Fix a sufficiently small constant $c \coloneqq c(\eps)>0$, and an arbitrary vertex $v\in[n]$. 
Recall that if $np\gg \ln n$, then \whp\ $G_{n,p}$ is Hamiltonian --- see, e.g., \cite[Theorem 6.5]{FriKar}. 
 We observe that \whp\ $|N(v)|p=(1+o(1))np^2\gg \ln n$, due to Claim \ref{cl:degrees}. Therefore, \whp\ $G[N(v)]$ contains a Hamilton cycle.

We will also rely on the following fact that we discussed in Section \ref{sc:proof_shallow_1}: w.h.p., for every $u\in [n]\setminus(N(v)\cup\{v\})$, there exists a unique connected component $C(u)\subseteq H(u)$ of size at least $cn^{1/3}$. Given $u,w\notin N(v)\cup\{v\}$,  we say that $w$ is {\it $u$-friendly} if $C(u)\cap C(w)\neq \emptyset$. We note that being friendly is a symmetric relation, i.e.\ $w$ is $u$-friendly if and only if $u$ is $w$-friendly. We need the following claim regarding the number of friendly pairs.

\begin{claim}\label{cl:hampath}
    W.h.p.\ for every $u\not\in N(v)\cup \{v\}$ there are at least $n/\ln^4 n$ vertices $w\not\in N(v)\cup \{v\}$ which are $u$-friendly.
\end{claim}

The proof of this fact is based on a similar technique as the proof of Claim \ref{cl:giant_to_giant}, and does not introduce any novel ideas, so we postpone it to Appendix~\ref{appendix_claim}.

A Hamilton path $P\subseteq \mathbf{G}$ such that $P\rightarrow \mathbf{G}$ w.h.p.\ is constructed iteratively. In each step, we reveal more information about the random graph and extend $P$.

\paragraph{Step 1:}
Expose all edges adjacent to $v$ --- this determines $N(v)$. 
Fix $w_1\in N(v)$, and an ordering of all the vertices {\it outside of} $(N(v)\cup\{v\})\setminus \{w_1\}$ such that $w_1$ appears first. We refer to this ordering as the \emph{canonical ordering}. 

Then, expose all edges with an endpoint in $N(v)$.
As mentioned above, \whp\ there is a Hamilton cycle in $\mathbf{G}[N(v)]$, and the assertion of Claim \ref{cl:hampath} holds. 
Set $\ell\coloneqq |N(v)|$ and let $v_1v_2\ldots v_\ell=w_1$ be a Hamilton path in $N(v)$. Define the prefix of the path $P$ under construction as $P_1\coloneqq vv_1v_2\ldots v_\ell$. 

\paragraph{Step 2:}
For every $1\leq i\leq \lfloor n/\ln^{6} n\rfloor$, we iteratively extend the path $P_i$ to a path $P_{i+1}$ by adding a vertex $w_{i+1}$.
To this end, we introduce a partition of $X\coloneqq [n]\setminus (N(v)\cup \{v\})\cup \{w_1\}$ into two sets $U_i$ and $W_i$. Throughout, we make sure that $|U_i|\leq i $ for all $i$.
For the initial step $i=1$, we set $U_1 \coloneqq \emptyset$ and $W_1\coloneqq X$. 
At every step $i$, the vertex $w_{i+1}$ is chosen from the set $W_i$ and attached to the endpoint $w_i$ of $P_i$.

Suppose that $U_i, W_i,$ and $w_i$ are given at the beginning of step $i$. By the assertion of Claim \ref{cl:hampath} there are at least $n/\ln^4 n$ vertices $x\in X=W_1$ which are $w_i$-friendly. This and the assumption that $|U_i|\leq i \leq n/\ln^{6} n$ imply that there are at least $n/\ln^5 n$ vertices $x\in W_i$ which are $w_i$-friendly.

Expose the edges between $w_i$ and $W_i$, then let $W'$ be the set of $w_i$-friendly neighbours of $w_i$ in $W_i$. With probability $1-\exp(-\Theta(np/\ln^5n))$ we have $|W'|\geq np/\ln^6 n$.  Choose $w_{i+1}$ to be an arbitrary vertex in $W'$ and set $U_{i+1}=U_{i}\cup\{w_{i}\}$ and $W_{i+1}=X\setminus U_{i+1}$.

Each iteration of the process is independent of the others, as it depends only on the edges between $w_i$ and $W_i$. Thus, the process is successfully completed with probability at least
\[
	\biggl(1-\exp(-\Theta(np/\ln^5n))\biggr)^{\lfloor n/\ln^{6} n\rfloor}  \geq 1- \frac{n}{\exp(n^{2/3}/\ln^5 n)\ln^6 n}= 1-o(1).
\]

\paragraph{Step 3:}
Let the process successfully terminate in $s\coloneqq \lfloor n/\ln^{6} n\rfloor$ steps. We get a path $P_{s+1}$. Note that the edges with both endpoints in $[n]\setminus V(P_{s+1})$ have not yet been revealed. Thus, w.h.p.\ there exists a Hamilton cycle in $\mathbf{G}[[n]\setminus V(P_{s+1})]$. Find an edge $\{w_{s+1},w_{s+2}\}$ between the endpoint $w_{s+1}$ of $P_{s+1}$ and the Hamilton cycle, and, using the edges of this cycle, extend $P_{s+1}$ to a Hamilton path $P$ in $\mathbf{G}$.\\


Let $\mathbf{G}_0\sim G(n,\varepsilon n^{-1/3})$ be independent of $\mathbf{G}$ and let $\mathbf{G}'=\mathbf{G}\cup\mathbf{G}_0$. In order to conclude the proof, it suffices to prove that w.h.p.\ $P \to \mathbf{G}'$. The latter would follow from the fact that w.h.p.\ $P$ activates a subtree $T\subseteq\mathbf{G}'$ considered in Section \ref{sc:proof_shallow_1}  
 Omitting technicalities, it is easy to see that, in order to activate $\mathbf{G}'$, we can also use a tree that contains only $\mathbf{G}$-edges incident to $v$. In other words, we shall prove the following: (1) $P$ activates all edges incident to $v$ in $\mathbf{G}$; (2) for every $w\in [n]\setminus[N_{\mathbf{G}}(v)\cup\{v\}]$, the path $P$ activates an edge between $w$ and $C'(w)$, which is the giant component of $\mathbf{G}'[N_{\mathbf{G}}(v)\cap N_{\mathbf{G}'}(w)]$. We note that, for every $w\notin N(v)\cup\{v\}$, the set $C(w)$ is entirely inside $C'(w)$. 

Activation of all $\mathbf{G}$-edges incident to $v$ is straightforward. 
Next, we show that for every $2\leq i\leq s+1$ we may activate an edge between $w_i$ and $C(w_i)\subseteq C'(w_i)$.
This is done inductively, by showing that actually all edges between $w_i$ and $C(w_i)$ can be activated. For $i=2$, by construction $w_1\in C(w_2)$ and the edge $\{w_1,w_2\}$ is initially activated. Then, we may also activate all edges between $w_2$ and $C(w_2)$. For $i\geq 3$, suppose that we have activated all edges between $w_{i-1}$ and $C(w_{i-1})$. By construction, there is $z\in C(w_{i-1})\cap C(w_i)$ and the edge $\{w_{i-1},w_{i}\}$ is initially activated. By induction, we may activate the edge $\{w_{i-1},z\}$, and thus also the edge $\{w_i,z\}$. So we may activate all edges between $w_i$ and $C(w_i)$.

Finally, let $x\in [n]\setminus V(P_{s+1})$, and let $P^x$ be the subpath of $P$ starting from $v$ and ending at $x$. Using the standard multiple-exposure technique, we get that the graph $\mathbf{G}_0[N_{\mathbf{G}'}(x)\cap V(P^x)\setminus\{v\}]$ is connected with probability $1-\exp(n^{-\Theta(1)})$. Therefore, w.h.p., for every $x$, such a graph is connected. Since any vertex from $C(x)$ belongs to this graph, we get that w.h.p., for every $x$, all edges between $x$ and $C(x)$ can be activated, completing the proof. 

\subsection{Remaining questions}
\label{subsc:questions}

In this paper, we estimate the activation threshold for all trees that have diameter at most $n^{1/18-\varepsilon}$, up to a constant factor. The next natural step would be to generalise this result to all trees. Clearly, our main ingredient --- Proposition~\ref{prop: local to global} --- imposes a limitation due to the bound $L/100$ on the diameter of a tree, which prevents such a generalisation within our current framework. Nevertheless, we believe that this diameter restriction can potentially be omitted. We also believe that every tree that w.h.p. activates $G_{n,p}$ can be reduced to a tree with bounded diameter. An alternative route to refining the assertion of Theorem~\ref{th:general} could be to develop a purely combinatorial approach to prove lower bounds on the activation threshold; however, we did not manage to invent it. 

Actually, even if the diameter constraint in Proposition~\ref{prop: local to global} was lifted, it would not suffice to push the lower bound beyond $3/4^{4/3}\cdot n^{-1/3}$. Thus, it remains possible that the upper bound in Theorem~\ref{th:general} is not optimal. We believe that $3/4^{4/3}\cdot n^{-1/3}$ is a sharp threshold for the property that, for every cycle in $G_{n,p}$, there is a subgraph in $G_{n,p}$, which is isomorphic to a triangulation of that cycle. Note that this property implies simple connectivity of $X^{(2)}(G_{n,p})$, yet it does not guarantee the existence of a single tree $T\subset G_{n,p}$ that activates the entire graph. Indeed, while triangulation of every cycle ensures that each cycle can be activated from some tree, it does not imply the existence of a common tree that activates all cycles. 

The above discussion motivates the following question: do hitting times for simple connectivity and for the property of being activated from a tree coincide? We suspect they may not and that, specifically, the (sharp) threshold for the former property is $3/4^{4/3}\cdot n^{-1/3}$, while the (sharp) threshold for the latter property is $n^{-1/3}$. 

Another interesting direction is to explore the behaviour of $\mathrm{wsat}(G_{n,p},F)$ for other graphs $F$. In particular, Kor\'{a}ndi and Sudakov~\cite{KorSud} established the stability of weak saturation~\eqref{eq:stability} for all cliques $F=K_s$. It is also known~\cite{Bidgoli} that the stability property~\eqref{eq:stability} has a threshold probability 
$$
p_{K_s}\in\left(n^{-\frac{2}{s+1}+o(1)},n^{-\frac{1}{2s-3}+o(1)}\right).
$$
Determining the exact exponent of $n$ in this threshold for any $s\geq 4$ remains a challenging open problem. It does not seem plausible that our topological approach can be used for this task. Nevertheless, it may still be useful for estimating the stability threshold when $F$ is a cycle.

\section*{Acknowledgements} 

The last author would like to thank Stepan Alexandrov and Nikolay Bogachev for their excellent and insightful introduction to the theory of hyperbolic metric spaces. This work has been supported by the Royal Society International Exchanges grant IES$\backslash$R3$\backslash$233177, GA\v{C}R grant 25-17377S, and ERC Synergy Grant DYNASNET 810115.

\appendix
\section{Proofs of Lemma \ref{lem: no ell-cycle between log and Klog LM} and Lemma \ref{lem: exists an ell-cycle with more than log LM}}\label{Appendix}

Here we prove both Lemma \ref{lem: no ell-cycle between log and Klog LM} and Lemma \ref{lem: exists an ell-cycle with more than log LM}.
Let us start with the following analogue of Lemma~\eqref{lem: expectation of the number of activations}.

\begin{lemma}\label{lem: expectation of the number of activations LM Appendix}
Suppose that $3\leq \ell=\ell(n)$, $v=v(n)$, and $w=w(n)$ are sequences of non-negative integers such that $w\leq 2v^2$.  
 Then, for large enough $n$, 
        \[
        \E\left[Z_{v,w}^{\ell}\right] \leq  n^{(-1/2+o(1))w+(1/2+o(1))\ell+1} \cdot v^{6w}\cdot(2^7 c^2)^{v}.
        \]
\end{lemma}

\begin{proof}  
For every $\mathbf{A}$-activation diagram $D\in P_{v,w}$, where $\mathbf{A}=(H;\Delta_1,\ldots,\Delta_s)$ is a nice activation process of an $\ell$-cycle, Lemma \ref{lem: number of edges in minimal process} and the two equalities $2d_H+\exc(\mathbf{A})=w$, and $|F(D)|=|E(D)|-|V(H)|-d_H+1$, which hold by the definition of a diagram, imply that 
    \[
    |F(D)|\geq 2v+2d_H +\exc(\mathbf{A})-2-\ell = 2v+w-2-\ell.
    \]
Hence, by Lemma \ref{lem: low w},
 \begin{align*}
    \E\left[Z_{v,w}^{\ell}\right]&=\sum_{D\in P_{v,w}}\mathbb{P}(D\subseteq \mathbf{H})= \sum_{D\in P_{v,w}}p^{|F(D)|} \leq  |P_{v,w}|\cdot p^{2v+w-2-\ell}\\
    &\leq  n^v (2v)^{4w} w^w \cdot \binom{2v+w-2}{w}2^{7v} \cdot \left(cn^{-\frac{1}{2}}\right)^{2v+w-2-\ell}\\
    &\leq (2v)^{4w} w^w \cdot  \binom{2v+w-2}{w} \cdot \left(cn^{-\frac{1}{2}}\right)^{w-2-\ell} (2^7c^2)^{v}.
 \end{align*}
 It remains to apply the inequality
\begin{equation*}
 \binom{2v+w-2}{w} \leq \begin{cases} (3v)^{w} \quad&\text{ if }w\leq v,\\ 2^{3w} \quad & \text{ if }w>v.
\end{cases}
\end{equation*}
Indeed, if $w\leq v$ then
\begin{align*}
       \E\left[Z_{v,w}^{\ell}\right]&\leq (2v)^{4w}w^w \cdot (3v)^{w} \cdot (cn^{-1/2})^{w-2-\ell} \cdot   (2^7 c^2)^{v}\\
       &\leq  n^{(-1/2+o(1))w+(1/2+o(1))\ell+1} \cdot v^{6w} \cdot(2^7 c^2)^{v}.
\end{align*}
Moreover, if $v<w\leq 2v^2$ then
\begin{align*}
       \E\left[Z_{v,w}^{\ell}\right]&\leq (2v)^{4w}w^w \cdot 2^{3w} \cdot (cn^{-1/2})^{w-2-\ell} \cdot(2^7 c^2)^{v}\\
       &\leq n^{(-1/2+o(1))w+(1/2+o(1))\ell+1} \cdot v^{6w} \cdot (2^7 c^2)^{v}.\qedhere
\end{align*}

\end{proof}

For the sake of convenience, we reiterate the statements of Lemma \ref{lem: no ell-cycle between log and Klog LM} and Lemma \ref{lem: exists an ell-cycle with more than log LM} below, and then prove them. 

We first recall \eqref{eq:Pvw=0} which states that $P_{v,w}=\emptyset$ for $w>2v^2$.
Thus, as in Section \ref{sec:main proof} we have $Z^{\ell}_{v,w} \equiv 0$ deterministically, whenever $w>2v^2$.

\xyz*

\begin{proof}
Recall that an activation process of an $\ell$-cycle $\mathbf{A}=(H;\Delta_1,\ldots ,\Delta_s)$ defines an activation diagram.
    If further $\mathbf{A}$ is a nice process, with $|V(K(\mathbf{A}))|=v$, and with $2d_F+\exc({\mathbf{A}})=w$, then the activation diagram belongs to $P_{v,w}$.
    Thus, in order to conclude the proof, it suffices to show that w.h.p.\ for any $K_1\ell\ln n\leq v\leq K_2\ell\ln n$ and any $w\leq 2v^2$ there is no $D\in P_{v,w}$ such that $D\subseteq \mathbf{H}$.
We will show that, for all $n$ large enough and for every $3\leq\ell\leq n^{\varepsilon}$, with probability at least $1-n^{-1/2}$,
$$
Z_{\ell}\coloneqq \sum_{v=K_1\ell\ln n}^{K_2\ell\ln n}\sum_{w=0}^{2v^2} Z_{v,w}^{\ell}=0.
$$
This clearly suffices to finish the proof by the union bound.
As $Z_\ell$ is an integer-valued random variable, by Markov's inequality, it is enough to show that $\E[Z_{\ell}]\leq n^{-1/2}$ for all $\ell$ in the range. We write $\E[Z_\ell]$ in the following way:
\begin{align*}
	\E[Z_\ell]=  \sum_{v=K_1\ell\ln n}^{K_2\ell\ln n}\sum_{w=0}^{2v^2} \E\left[Z_{v,w}^{\ell}\right]. 
\end{align*}
By applying Lemma \ref{lem: expectation of the number of activations LM Appendix} in conjunction with our assumption that $\eps<1/12$ and that $\ell\leq n^{\eps}$ we obtain the following inequality:
\begin{align*}
        \sum_{v=K_1\ell\ln n}^{K_2\ell\ln n}\sum_{w=0}^{2v^2} \E\left[Z_{v,w}^{\ell}\right]&\leq \sum_{v=K_1\ell\ln n}^{K_2\ell\ln n} \sum_{w=0}^{2v^2} n^{(-1/2+6\eps+o(1))w+(1/2+o(1))\ell+1} \cdot(2^7 c^2)^{v}\\
            &\leq \exp((1+\ell/2+o(\ell))\ln n) \cdot(2^7 c^2)^{K_1\ell \ln n}.
    \end{align*}
By our assumptions that $\eps<1/12$ and that $K_1>-\frac{1}{\ln(2^7c^2)}$ we obtain the required inequality for large enough $n$:
\[
    \E[Z_\ell] \leq \exp((1+\ell/2)\ln n) \cdot(2^7 c^2)^{K_1\ell \ln n}<\exp((1-\ell/2)\ln n)\leq n^{-1/2}, 
\]
completing the proof.
\end{proof}

\ab*

\begin{proof}
Let $X_{\ell}$ denote the random variable counting the number of $\ell$-cycles $C\subseteq K_n$ that \textbf{do not} satisfy the assertion of the lemma. Formally, $X_{\ell}$ counts the number of $\ell$-cycles $C\subseteq K_n$ such that:
\[
	\exists H \subseteq K_n \text{ with } |V(H)|\in [\ell, K\ell\ln n] \text{ and satisfying }  H\stackrel{\mathbf{H}}{\longrightarrow} C \text{ via a nice activation process}.
\]
In order to conclude the proof of the lemma, it is enough to show that w.h.p., for every $\ell\leq n^{\eps}$, 
\[
	X_{\ell} < (n/2)^{\ell}.
\]
In what follows we prove that
 $
 \E[X_{\ell}] \leq  n^{5\ell/6} = (n/2^6)^{-\ell/6}\cdot (n/2)^{\ell}.
 $
 Then, Markov's inequality and the union bound over $\ell$ yield the desired assertion.

The rest of the proof resembles the proof of Lemma \ref{lem: no ell-cycle between log and Klog LM}. 
Recall that an activation process $\mathbf{A}=(H;\Delta_1,\ldots ,\Delta_s)$ of an $\ell$-cycle defines an activation diagram.
    If $\mathbf{A}$ is nice, $|V(K(\mathbf{A}))|=v$, and $2d_F+\exc({\mathbf{A}})=w$, then the activation diagram belongs to $P_{v,w}$.
    Thus, 
$$
\E[X_{\ell}]\leq \sum_{v= \ell}^{K\ell\ln n}\sum_{w=0}^{2v^2} \E[Z_{v,w}^{\ell}].
$$
By applying Lemma \ref{lem: expectation of the number of activations LM Appendix} in conjunction with our assumptions that $\eps < 1/12$, $c<2^{-7/2}$, and that $\ell\leq n^{\eps}$ we obtain the following inequality:
\begin{align*}
         \sum_{v= \ell}^{K\ell\ln n}\sum_{w=0}^{2v^2} \E[Z_{v,w}^{\ell}]&\leq \sum_{v= \ell}^{K\ell\ln n}\sum_{w=0}^{2v^2} n^{(-1/2+6\eps+o(1))w+(1/2+o(1))\ell+1}\leq n^{1+(1/2+o(1))\ell}.
\end{align*}
Then, for $\ell\geq 4$, we have $1+\ell/2<5\ell/6$, implying
\[
     \E[X_{\ell}] \leq n^{1+(1/2+o(1))\ell}\leq n^{5\ell/6},
\]
and completing the proof.
\end{proof}

\section{Proof of Claim \ref{cl:hampath}}\label{appendix_claim}

First expose the edges incident to $v$ and notice that by Claim \ref{cl:degrees}, w.h.p.\ $|N(v)|=np(1+o(1))$ and thus w.h.p.\ $|[n]\setminus N(v)|=(1+o(1))n$. We further assume that these equalities hold deterministically. By the union bound it is enough to show that with probability $1-o(n^{-1})$, for every fixed $u\not \in N(v) \cup \{v\}$ there are at least $n/\ln^4 n$ vertices $w\not\in N(v)\cup\{v\}$ which are $u$-friendly. 

Fix $u\not \in N(v) \cup \{v\}$ and let $\{A_1,\ldots ,A_{n/\ln^4 n}\}$ be a partition of $[n]\setminus (N(v)\cup\{v,u\})$ into sets of size $\tau\coloneqq \lfloor \ln^3 n\rfloor$ each.
 To conclude the proof, by the union bound it is enough to show that each set $A_i$ contains a $u$-friendly vertex with probability at least $1-o(n^{-2})$.
Indeed, fix $A_k$, and let $\{u_i\}_{i=1}^{\tau}$ be an ordering of $A_k$.
For every $i=1,\ldots, \tau$, we expose sequentially some sets of edges `associated' with $u_i$. These sets of edges are defined in the following way. Assume that, before step $i\in[\tau]$, all edges between $\{u\eqqcolon u_0,u_1,\ldots,u_{i-1}\}$ and $N(v)$ are exposed. Also, for every $j\in\{0,\ldots,i-1\}$, the edges in a certain set $U_j\subseteq N(u_j)\cap N(v)$, that is defined below, are exposed. Set $U_0=N(u)\cap N(v)$.

At step $i$, we first expose the adjacency relations between $u_i$ and $N(v)$. Set $\mathcal{E}_i$ to be the event that $u_i$ has a neighbour in $C(u)$.
If $\mathcal{E}_i$ is not satisfied, set $U_i=\emptyset$. Otherwise, set
$$
U_i:=(N(u_i)\cap N(v))\setminus (U_0\sqcup U_1\sqcup\ldots\sqcup U_{i-1})
$$ 
and expose the edges of $\mathbf{G}[U_i]$. Also, let $z\in V(C(u))$ be a neighbour of $u_i$. Subject to the event that \begin{equation}
\label{eq:common_neighbourhood_exposed2}
|N(v)\setminus(U_0\sqcup U_1\sqcup\ldots\sqcup U_{i-1})|=(1+o(1))np,
\end{equation}
the set $U_i$ satisfies
\begin{equation}
\label{eq:U_i_cardinality2}
|U_i|=(1+o(1))np^2
\end{equation}
 with probability $1-o(n^{-3})$ by the Chernoff bound. Due to Theorem~\ref{th:giant}, with probability $1-o(n^{-3})$, there exists a component of size at least $cn^{1/3}$ in $\mathbf{G}[U_i]$. If such a component exists, then it has a neighbour of $z$ with probability at least 
\begin{equation}
\label{eq:two_neighbours2}
1-(1-p)^{cn^{1/3}}\sim 1-e^{-(1+\varepsilon)c}.
\end{equation}
As soon as we find such a neighbour in $U_i$, we get the desired $u$-friendly vertex, namely $u_i$.

Let $i^*$ be the first $i\in[\tau]$ such that at least $\ln^2 n$ events $\mathcal{E}_j$, $j\leq i$, hold; if no such $i$ exists let $i^*=\tau+1$.
Since the events $\mathcal{E}_i$, $i\in[\tau]$, are independent and  
$$
\mathbb{P}(\mathcal{E}_i)=1-(1-p)^{|C(u)|}=\Theta(1)
$$
for all $i\in[\tau]$, we get that $i^*\leq \tau$ with probability $1-o(n^{-2})$. Due to Claim~\ref{cl:degrees}, with probability $1-o(n^{-2})$, 
$$
(1-o(1))np-(1+o(1))\ln^2 n \cdot np^2 \leq |N(v)\setminus(U_0\sqcup U_1\sqcup\ldots\sqcup U_{i})|=(1-o(1))np
$$
for all $i\leq i^*$.
Due to~\eqref{eq:two_neighbours2} and Theorem~\ref{th:giant}, with probability $1-o(n^{-2})$, there exists $i\leq i^*$ such that $\mathcal{E}_i$,~\eqref{eq:common_neighbourhood_exposed2},~and~\eqref{eq:U_i_cardinality2} hold, the largest component in $U_i$ is inside $C(u_i)$, and this component has a neighbour of $z$, completing the proof.

\end{document}